\documentclass[12pt]{amsart}

\usepackage{amsmath,amssymb,amsfonts,amsthm,mathrsfs}
\usepackage[T1]{fontenc}  
\usepackage{palatino}     
\usepackage{xfrac}
\usepackage[usenames,dvipsnames]{xcolor}
\usepackage{tikz}
\usetikzlibrary{calc,shapes.geometric}

\usepackage{mathtools}

\def\mid_firstreturn{
  \draw (3, 0) -- (3, 2) -- (4, 2) -- (4, 3);
  \draw (3, 3) -- (1, 3);
  \draw (0, 3) -- (0, 2) -- (1, 2) -- (1, 0);
  \draw[color=gray,dashed] (1, 0) -- (3, 0);
}

\def\topfirstreturn{
  \draw (3, 0) -- (3, 2) -- (4, 2) -- (4, 3) --
        (0, 3) -- (0, 2) -- (1, 2) -- (1, 0);
  \draw[color=gray,dashed] (1, 0) -- (3, 0);
}

\newcommand{\firstreturnThelper}[1]{
  \ifnum #1>0
    \mid_firstreturn
  
    \edef\level{#1}
    \pgfmathparse{int(\level-1)}
    \edef\level{\pgfmathresult}
    \begin{scope}[xshift=-0.5cm, yshift=3cm]
      \begin{scope}[scale=0.5]
        \firstreturnThelper{\level}
      \end{scope}
    \end{scope}
    \begin{scope}[xshift=2.5cm, yshift=3cm]
      \begin{scope}[scale=0.5]
        \firstreturnThelper{\level}
      \end{scope}
    \end{scope}
  \else
    \topfirstreturn
  \fi
}

\def\topT{
  \draw (3, 0) -- (3, 2) -- (4, 2) -- (4, 3) --
        (0, 3) -- (0, 2) -- (1, 2) -- (1, 0);
}
      
\def\midT{
  \draw (3, 0) -- (3, 2) -- (4, 2) -- (4, 3);
  \draw (3, 3) -- (1, 3);
  \draw (0, 3) -- (0, 2) -- (1, 2) -- (1, 0);
}

\newcommand{\fractalThelper}[1]{
  \ifnum #1>0
    \midT
  
    \edef\level{#1}
    \pgfmathparse{int(\level-1)}
    \edef\level{\pgfmathresult}
    \begin{scope}[xshift=-0.5cm, yshift=3cm]
      \begin{scope}[scale=0.5]
        \fractalThelper{\level}
      \end{scope}
    \end{scope}
    \begin{scope}[xshift=2.5cm, yshift=3cm]
      \begin{scope}[scale=0.5]
        \fractalThelper{\level}
      \end{scope}
    \end{scope}
  \else
    \topT
  \fi
}

\newcommand{\labelT}[3]{
  \ifnum #1>0
  
    \edef\level{#1}
    \pgfmathparse{int(\level-1)}
    \edef\level{\pgfmathresult}
    \edef\scale{#3}
    \pgfmathparse{int(\scale+1)}
    \edef\scale{\pgfmathresult}
    \begin{scope}[yshift=3cm]
      \begin{scope}[xshift=-0.5cm,scale=0.5]
        \labelT{\level}{#2 1}{\scale}
      \end{scope}
      \begin{scope}[xshift=2.5cm,scale=0.5]
        \labelT{\level}{#2 0}{\scale}
      \end{scope}
    \end{scope}
  \else
    \edef\scale{#3}
    \pgfmathparse{0.5^(\scale-1)}
    \edef\scale{\pgfmathresult}
    \node[scale=\scale] at (2, 1) {$#2$};
    \draw[dashed] (1, 0) -- (3, 0);
  \fi

}

\newcommand{\fractalT}[1]{
  \fractalThelper{#1}
  \draw (1, 0) -- (3, 0);
}

\newcommand\labeledT[1]{
  \foreach \n in {1, 2, ..., #1} {
    \labelT{\n}{}{0}
  }
  \node[scale=2] at (2, 1) {$\epsilon$};
}

\def\xs{0.5in}
\def\ys{0.375in}
\def\iterates{8}

\newcommand{\tfractal}{T_\infty}
\newcommand{\tfractalr}{\mathring{T}_\infty}
\newcommand{\tsurf}{\mathcal{T}}
\newcommand{\tsurfr}{\mathring{\mathcal{T}}}

\DeclareMathOperator{\dist}{dist}

\def\N{\mathbb{N}}

\def\esurf{\mathcal{E}}

\def\Q{\mathcal{Q}}
\def\Qr{\mathring{\mathcal{Q}}}
\renewcommand{\phi}{\varphi}

\newtheorem{theorem}{Theorem}[section]
\newtheorem{lemma}[theorem]{Lemma}
\newtheorem{proposition}[theorem]{Proposition}
\newtheorem{corollary}[theorem]{Corollary}

\theoremstyle{definition}
\newtheorem{remark}[theorem]{Remark}
\newtheorem{definition}[theorem]{Definition}

\title{The Wild, Elusive Singularities of the $T$-fractal Surface}
\author{Chris Johnson}
\address{Department of Mathematics and Computer Science, Western
  Carolina University, Cullowhee, NC, USA}
\email{cjohnson@wcu.edu}
\author{Robert Niemeyer}
\address{Department of Mathematics and Statistics, Metropolitan State
  University of Denver, Denver, CO, USA}
\email{niemeye1@msudenver.edu}
\date{\today}


\begin{document}
\begin{abstract}
  We give a rigorous definition of the $T$-fractal translation
  surface, and describe some its basic geometric and dynamical
  properties.  In particular, we study the singularities attached to
  the surface by its metric completion and show there exists a Cantor
  set of ``elusive singularities.''  We show these elusive
  singularities can be thought of as a generalization of the wild
  singularities introduced by Bowman and Valdez \cite{BowmanValdez}.
  In particular, we show that every elusive singularities has an
  infinite discrete set of rotational components.
\end{abstract}
\maketitle

\section{Introduction}
\label{sec:intro}

The $T$-fractal billiard has been studied in \cite{LapNie1,LapMilNie}.
In \cite{LapNie1}, one is introduced to the $T$-fractal translation
surface, but details of the construction are not given, especially on
the nature of the elusive singularities.  Such points are called
``elusive'' for the fact that they exist only in the limiting object.
Periodic orbits of the $T$-fractal billiard are described in
\cite{LapNie1} and further results on the nature of orbits of the
$T$-fractal billiard are given in \cite{LapMilNie}, and, most notably,
what is called a nontrivial path with an `irrational direction'
reaching an elusive point of the billiard in a manner which is highly
consistent with an orbit with an initial `rational
direction.'\footnote{`Rational direction' (resp., `irrational
  direction') means that for an initial direction $\theta$ (measured
  relative to some side of the $T$-fractal billiard),
  $\tan\theta\in\mathbb{Q}$ (resp.,
  $\tan\theta\in\mathbb{R}\setminus\mathbb{Q}$).}

In this paper we precisely describe the geometry of the $T$-fractal
translation surface.  Such a surface has finite area and is obtained
by systematically gluing sides of scaled copies of what we have termed
the \emph{quad-T surface}.
Such a surface is constructed from four copies of the first-level
approximation of the $T$-fractal billiard. Given its appearance and
intimate connection to the fractal billiard table, such a surface has
been called a \emph{fractal translation surface.}  We provide a
rigorous study of the singularities of the $T$-fractal translation
surface attached via its metric completion.  These singularities come
in two types: a finite angle conical singularity with cone angle
measuring $6\pi$, which comes from a corner of the fractal billiard
table, and a wild singularity, the set of which does not appear in any
quad-T subsurface, these being referred to as elusive singularities of
the $T$-fractal translation surface.  As we will show in Section
\ref{sec:theMetricCompletionOfTsurf} and state in Proposition
\ref{prop:elusive_cantor}, the set of elusive singularities of the
$T$-fractal translation surface is a Cantor set and the geodesic loops
about finite angle singularities are decreasing in radius.  Such a
geometry thus complicates any attempt to discuss the geodesic flow on
the metric completion.  We will argue that elusive singularities are
\emph{wild singularities}, this being a term introduced in
\cite{BowmanValdez} and that the definition of such can be adjusted to
account for singularities of a fractal translation surface.  Indeed,
in Theorem \ref{thm:every_elusive_linear_approach}, we show that every
elusive singularity has a linear approach, yet the rotational
components of an elusive singularity with irrational address cone
angle zero, as stated in Theorem~\ref{thm:positive_angle_rational} and
Corollary~\ref{cor:ae_rotational_zero_angle}.  However, rotational
components of elusive singularities with rational address may have
positive cone angle.


The paper is organized as follows.  In Section \ref{sec:background},
we discuss the necessary background material on translation surfaces
and polygonal billiards via the first-level approximation of the
$T$-fractal billiard.  We briefly discuss the geometric and analytic
properties of the singularities of a translation surface and provide
an example of a wild singularity (as defined in \cite{BowmanValdez}).
In Section \ref{sec:constructionTheT-FractalFlatSurface}, the
$T$-fractal surface is introduced and formally defined.  While the
idea of such a surface has been discussed in \cite{LapMilNie}, this
paper constitutes the first formal definition of such a surface that
is amenable to rigorous analysis. As previously mentioned, such a
surface will be built from the so-called quad-T surface shown in
Figure \ref{fig:quadt}, this being the focus of Section
\ref{sec:constructionTheT-FractalFlatSurface}.  In
Section~\ref{sec:IETs} we discuss a special interval exchange on the
T-fractal surface which we then use in
Section~\ref{sec:linearApproachesExist} to show that there exists a
linear approach to each elusive singularity of the surface.  In
Section~\ref{sec:theMetricCompletionOfTsurf} we make several
observations about the metric geometry of the $T$-fractal surface
which we use in
Section~\ref{sec:elusiveSingularitiesWildSingularities} to prove that
each elusive singularity is a ``wild singularity,'' mildly
generalizing the notion introduced by Bowman and Valdez.  We conclude the paper with a brief discussion in
Section \ref{sec:finalRemarks}.


\bigskip A glossary of notation is provided in Table
\ref{tbl:notation} so that the reader may more easily determine where
a term was first defined.

\begin{table}
\begin{tabular}{|cp{7.5 cm}c|}	
	\hline
	Notation & \centering{Explanation} & Page reference\\
	\hline
	$\mathring{T}_n$ & The $n$th level $T$-fractal approximation with corners removed& \pageref{ntn:Tnring}\\
	$T_n$ & The $n$th level approximation of the $T$-fractal  & \pageref{ntn:Tn}\\
		$\tfractalr$ & The $T$-fractal with corners removed and no elusive points & \pageref{ntn:tfractalr}\\
	$\tfractal$ & The $T$-fractal& \pageref{ntn:tfractal}\\
	$\tsurfr$ &The $T$-fractal translation surface (no corners and no elusive singularities) & \pageref{ntn:tsurfr} \\
	$\tsurf$ &The metric completion of $\tsurfr$ (contains corners and elusive singularities) & \pageref{ntn:tsurf}\\
	$\Q$ & The quad-$T$ surface& \pageref{ntn:Q}\\
	$\Q_s$ & A quad-$T$ subsurface of $\tsurf$; $\Q$ scaled by $2^{-|s|}$ & \pageref{ntn:Qs} \\
	$\Qr$ & The quad-$T$ surface with corners removed (boundary components remain) & \pageref{ntn:Qr}\\
	$\Qr_s$ &The quad-$T$ subsurface of $\tsurfr$ & \pageref{ntn:Qrs}\\
		$\epsilon$ & The empty string & \pageref{ntn:epsilon} \\
		$\sigma_i^s$ & A boundary component of $\Q^s$ &\pageref{ntn:sigmais}\\
	$\gamma_i^s$ & A boundary component of $\Q^s$ &\pageref{ntn:gammais} \\
	$\mathscr{B}$ & The set $\{0,1\}$ & \pageref{ntn:scriptB}\\
	$\mathscr{B}^*$ & The set of all finite sequences of $0$s and $1$s & \pageref{ntn:scriptBstar}\\
	$\mathcal{B}^s$ & A branch of  $\tsurf$ & \pageref{ntn:branch} \\
	$s\wedge t$ & The longest substring common to both $s$ and $t$ & \pageref{ntn:wedge}\\
	$\esurf$ & The set of elusive singularities &\pageref{ntn:elusiveSingularity}\\
	$\sigma(x)$ & A map $\sigma : \mathcal{B}^\epsilon \to \mathscr{B}^*$ & \pageref{ntn:sigmax}\\
	$\alpha(x)$ & The address of the elusive singularity $x$ & \pageref{ntn:alphax} \\
	$\Gamma$ & The collection of boundary components $\gamma_i$ of $\Q$ & \pageref{ntn:Gamma}\\
	$\Sigma$ & The collection of boundary components $\sigma_i$ of $\Q$ & \pageref{ntn:Sigma}\\
	$F_\theta$ & An interval exchange transformation & \pageref{ntn:FthetaFIET}\\
	$\Phi$ & The renormalization map & \pageref{ntn:PhiRenormalizationMap}\\
	$\nabla(s)$ & The sequence $s$ with the right most bit of $s$ removed & \pageref{ntn:nablas}\\
	$\lambda(s) $ & The right-most bit of $s$ & \pageref{ntn:lambdas} \\
	\hline
\end{tabular}
\caption{A glossary of notation.}
\label{tbl:notation}
\end{table}

\section{Background}
\label{sec:background}
In this section, we recall some necessary background about translation
surfaces and polygonal billiards, as well as establish notational
conventions that will aid us in our analysis and
discussion. Additionally, we describe the $T$-fractal billiard,
setting the stage for Section
\ref{sec:constructionTheT-FractalFlatSurface}.

\subsection{Translation surfaces}
There are many equivalent definitions of a translation surface, but
for the majority of this paper, what is perhaps the simplest
definition will suffice.\footnote{When necessary, we will make use of
  the definition of a translation surface given in terms of
  holomorphic $1$-form on a Riemann surface.}

\begin{definition}
  A \emph{translation surface} $\mathring{X}$ is a surface equipped
  with an atlas of charts where all chart changes are accomplished by
  translations.
\end{definition}


It is known that every translation surface can be equipped with a
natural metric, which is typically incomplete.  We will maintain the
convention in the literature by using $\mathring{X}$ for the initial
surface, and $X$ for its completion.\footnote{$\mathring{X}$
  represents a surface which may have some ``holes'' in it and we use
  the ring above the letter to indicate such.}

One way to construct a translation surface is to consider a collection
of polygons in the plane with corners removed, and edges identified in
pairs such that each edge of a polygon is glued to a parallel
edge of the same length by translation, subject to the condition that
the inward-pointing normal vectors along these edges point in opposite
directions.  

  



Consider the $T$-shaped polygon---hereafter denoted by $T_0$ and $\mathring{T}_0$ when corners are removed---given
in Figure~\ref{fig:tpolygon}.  Four copies of $\mathring{T}_0$ can be arranged
with opposite, parallel sides identified so as to construct a translation
surface, as shown in Figure \ref{fig:t-shapedPolygonFlatSurface}.  The 
\emph{$n$-th level approximation of the $T$-fractal}, denoted \label{ntn:Tn} $T_n$,
is obtained from $T_{n-1}$ by attaching $2^n$ copies of $T_0$, each
scaled by $2^{-n}$, to the top left- and right-hand portions of the
scaled copies of $T_0$ sitting at the top of $T_{n-1}$.  See
Figure~\ref{fig:tfractalIterativeConst} for the cases of $T_0$ through
$T_3$. We denote by \label{ntn:Tnring} $\mathring{T}_n$  the $n$th level approximation of the $T$-fractal billiard $T_n$ with corners removed.  Consequently, \label{ntn:tfractalr} $\tfractalr$ is the $T$-fractal billiard with corners removed (and no elusive points).

  \begin{figure}
    \centering
    \begin{tikzpicture}[scale=2]
      \draw
      (0, 0) --
      (1, 0) --
      (1, 1) --
      (1.5, 1) --
      (1.5, 1.5) --
      (-0.5, 1.5) --
      (-0.5, 1) --
      (0, 1) -- cycle;

      \draw[|<->|] (0, -0.125) -- node[midway, fill=white] {$1$} (1,
      -0.125);
      \draw[|<->|] (1.125, 0) -- node[midway, fill=white] {$1$}
      (1.125, 1);
      \draw[|<->|] (1.625, 1) -- node[midway, fill=white]
      {$\frac{1}{2}$} (1.625, 1.5);
      \draw[|<->|] (-0.5, 1.625) -- node[midway, fill=white] {$2$}
      (1.5, 1.625);
    \end{tikzpicture}
    \caption{The T-shaped polygon, $T_0$. When corners are removed, the set is denoted by $\mathring{T}_0$.}
    \label{fig:tpolygon}
  \end{figure}

\begin{figure}[h!]
  \centering
  \begin{tikzpicture}[scale=0.8]
    \begin{scope}[thick,xshift=1.25in]
      
      \draw (0, 0) --node[left]{$h$} (0, 2) 
                   --  node[below]{$g'$}(-1, 2)
                   --  node[left]{$f$}(-1, 3);
      \draw (0, 3) --node[above]{$e'$} (2, 3);
      \draw (3, 3) --node[right]{$d$} (3, 2)
                   --  node[below]{$c'$}(2, 2)
                   --  node[right]{$b$}(2, 0)--node[below]{$a'$}(0,0);

      \draw (-1, 3) --  (0, 3);
      \draw (2, 3) --  (3, 3);
      \draw (2, 0) --  (0, 0);
    \end{scope}

    \begin{scope}[thick,xshift=-1.25in]
       
      \draw (0, 0) --node[left]{$b$} (0, 2) 
                   --  node[below]{$c$}(-1, 2)
                   --  node[left]{$d$}(-1, 3);
      \draw (0, 3) --node[above]{$e$} (2, 3);
      \draw (3, 3) --node[right]{$f$} (3, 2)
                   --  node[below]{$g$}(2, 2)
                   --  node[right]{$h$}(2, 0)--node[below]{$a$}(0,0);

      \draw (-1, 3) -- (0, 3);
      \draw (2, 3) -- (3, 3);
      \draw (2, 0) --  (0, 0);
    \end{scope}

    \begin{scope}[thick,xshift=-1.25in,yscale=-1, yshift=0.5in]

      \draw (0, 0) --node[left]{$b'$} (0, 2) 
                   --  node[above]{$c$}(-1, 2)
                   --  node[left]{$d'$}(-1, 3);
      \draw (0, 3) --node[below]{$e$} (2, 3);
      \draw (3, 3) --node[right]{$f'$} (3, 2)
                   --  node[above]{$g$}(2, 2)
                   --  node[right]{$h'$}(2, 0)--node[above]{$a$}(0,0);
      \draw (-1, 3) --  (0, 3);
      \draw (2, 3) --  (3, 3);
      \draw (2, 0) -- (0, 0);
    \end{scope}
    
    \begin{scope}[thick,xshift=1.25in,yscale=-1, yshift=0.5in]
             
      \draw (0, 0) --node[left]{$h'$} (0, 2) 
                   --  node[above]{$g'$}(-1, 2)
                   --  node[left]{$f'$}(-1, 3);
      \draw (0, 3) --node[below]{$e'$} (2, 3);
      \draw (3, 3) --node[right]{$d'$} (3, 2)
                   --  node[above]{$c'$}(2, 2)
                   --  node[right]{$b'$}(2, 0)--node[above]{$a'$}(0,0);
      \draw (-1, 3) --  (0, 3);
      \draw (2, 3) --  (3, 3);
      \draw (2, 0) --  (0, 0);
    \end{scope}
  \end{tikzpicture}
  \caption{Four copies of the T-shaped billiard table $\mathring{T}_0$---corners removed, but not shown as such---with opposite
    and parallel sides identified results in a translation surface.}
  \label{fig:t-shapedPolygonFlatSurface}
\end{figure}

  \begin{figure}[h!]
    \begin{center}
      \begin{tikzpicture}
        \begin{scope}[scale=0.5]
          \begin{scope}[xshift=-3in]
            \fractalT{0}
          \end{scope}
          \begin{scope}[xshift=-1in]
            \fractalT{1}
          \end{scope}
          \begin{scope}[xshift=1.25in]
            \fractalT{2}
          \end{scope}
          \begin{scope}[xshift=3.75in]
            \fractalT{3}
          \end{scope}
        \end{scope}
      \end{tikzpicture}
      \caption{The iterative construction of the $T$-fractal billiard table.  Shown here are the prefractal approximations $T_0$, $T_1$, $T_2$ and $T_3$. When corners are removed, the notation is $\mathring{T}_n$, $n=0,1,2,3$.}
      \label{fig:tfractalIterativeConst}
    \end{center}
  \end{figure}

\begin{definition}[$T$-fractal billiard table]
  We define the \emph{$T$-fractal billiard table}, denoted by \label{ntn:tfractal} $\tfractal$,
  as the union of all the $n$-th level approximations,
  \[
    T_\infty = \bigcup_{n = 0}^\infty T_n.
  \]
  See Figure~\ref{fig:tfractal} for an illustration of $\tfractal$.
\end{definition}

\begin{figure}[h!]
  \centering
  \begin{tikzpicture}
    \fractalT{8}
  \end{tikzpicture}
  \caption{The $T$-fractal billiard table, $\tfractal$.}
  \label{fig:tfractal}
\end{figure}

\begin{remark}
  How we have defined $\tfractal$ differs from how the second author
  has defined $\tfractal$ in previous joint papers, e.g.,
  \cite{LapNie1} and \cite{LapMilNie}.  Specifically, in previous
  papers, $\tfractal$ was defined to be the closure of
  $\bigcup_{n=0}^\infty T_n$ and the notation for the billiard table
  was actually $\Omega(\tfractal)$.  The change to simpler notation and
  an alternate description of $\tfractal$ facilitates our construction
  of the $T$-fractal translation surface shown in Figure
  \ref{fig:unfolded_identifications}.  
\end{remark}

%
%

In general, as chart changes are translations, any translation-invariant quantity
defined in the plane can be pulled back to $\mathring{X}$.  In
particular, we can define a measure on the surface by pulling back the
Lebesgue measure of the plane, and a metric by pulling back the
standard Euclidean metric.  This metric space will typically not be
complete, and we are concerned with how the geometry of the surface
extends to those points added by the metric completion, which we will
denote by $X$.

Translations in the plane preserve direction, and so the translation
surface $\mathring{X}$ comes with a well-defined notion of direction.
We may consider geodesic flow on the surface in any given direction,
though we may need to delete a subset of the surface for the flow to
be defined for all time. 

The points of $X \setminus \mathring{X}$ are called
\emph{singularities} of the translation surface $\mathring{X}$ and come in several types which we
may classify by considering the families of geodesics in
$\mathring{X}$ which approach points of $X \setminus \mathring{X}$.
The definitions below are equivalent to those of \cite{BowmanValdez},
but have been modified slightly to better suit the purposes of this
paper.
\begin{definition}[Linear approach]
 A \emph{linear approach} to $x \in X$ is an injective map
 $\gamma : (0, \infty) \to \mathring{X}$ whose image is a geodesic
 segment in $\mathring{X}$ where
 $\displaystyle \lim_{t \to \infty} \gamma(t) = x$.  
\end{definition}

\begin{definition}[Directionally equivalent]
 We will say two
 linear approaches to $x$, $\gamma_1$ and $\gamma_2$, are
 \emph{directionally equivalent} if there exist values $a_1$ and
 $a_2$ such that the image of $(a_1, \infty)$ under $\gamma_1$ equals
 the image of $(a_2, \infty)$ under $\gamma_2$.  That is, the
 geodesic segments given by $\gamma_1$ and $\gamma_2$ approach the
 same point of $X$ from the same direction.  We will let $D[\gamma]$
 denote the directional equivalence class of a linear approach
 $\gamma$.
\end{definition}

Linear approaches to a given point can be divided into several
families where, intuitively, we can rotate one linear approach to $x$
to another, passing through linear approaches to $x$.  To make this
idea precise we must introduce the idea of a sector of a translation surface.
\begin{definition}[Standard sector]
 We will define \emph{the standard sector} of radius $r > 0$ and
 angle $\theta > 0$ as the translation surface $S_{r, \theta}$
 obtained by equipping the open strip
 $(-\log(r), \infty) \times (-\sfrac{\theta}{2}, \sfrac{\theta}{2})$,
 thought of as a subset of the complex plane $\mathbb{C}$, with the
 translation structure obtained by local integration of the 1-form
 $\omega = e^{-z} \, dz$.  We define the sector of angle $\theta = 0$
 and radius $r > 0$, $S_{r, 0}$, as the ray
 $(-\log(r), \infty) \times \{0\}$.  Local integration of
 $e^{-z} \, dz$ gives $S_{r, 0}$ the structure of a translation
 1-manifold.  We define the sector of infinite angle as the
 translation surface $S_{r, \infty}$ obtained by local integration of
 $e^{-z} dz$ in the open half-plane
 $(-\log(r), \infty) \times (-\infty, \infty) \subseteq \mathbb{C}$.
\end{definition}
\begin{definition}[Sector]
 A \emph{sector} of angle $0 \leq \theta \leq \infty$ and radius
 $r > 0$ in a translation surface $\mathring{X}$ is an isometry
 $\psi$ from the standard sector $S_{r, \theta}$ to an open subset of
 $\mathring{X}$.  We will say the sector is \emph{centered} at
 $x \in X$ if $\displaystyle \lim_{z \to \infty} \psi(z) = x$.  We
 will sometimes abuse language and refer to the image of $\psi$ as
 the sector.
\end{definition}

\begin{definition}[Rotationally equivalent]
 \label{def:rotationally_equivalent}
 We will say that two linear approaches $\gamma_1$ and $\gamma_2$ to
 $x \in X$ are \emph{rotationally equivalent} if there exists a
 sector $\psi : S_{r, \theta} \to \mathring{X}$ in $\mathring{X}$,
 centered at $x$, such that for some
 $y_k \in (-\sfrac{\theta}{2}, \sfrac{\theta}{2})$,
 $k = 1, 2$, the map $t \mapsto \psi(t + iy_k)$ is a linear approach
 to $x$ which is directionally equivalent to $\gamma_k$.
\end{definition}

\begin{definition}[Rotational component and cone angle]
 A \emph{rotational component} of $x \in X$ is a rotational
 equivalence class of linear approaches to $x$, and the supremum of
 all angles of sectors in $\mathring{X}$ containing linear approaches
 in that rotational component is the \emph{cone angle} of the
 rotational component.  If a point has only one rotational component,
 then we will sometimes refer to the cone angle of that rotational
 component as the cone angle of the point.
\end{definition}
\begin{definition}
  Let $x \in X \setminus \mathring{X}$ be a singularity of the translation surface $\mathring{X}$. Then, $x$ can be described as one of the following:
  \begin{description}
  \item[Removable singularity] We say that $x$ is a \emph{removable
      singularity} if it has only one rotational component, and that
    rotational component is
    isometric to a circle of circumference $2\pi$.  In this case, a
    neighborhood of $x$ is isometric to a disc in the plane.
  \item[(Finite angle) conical singularity] We say that $x$ is a
    (\emph{finite angle}) \emph{conical
      singularity} if it has only one rotational component, and that
    rotational component is
    isometric to a circle of circumference $2n\pi$ for some positive
    integer $n$.  In this case, a punctured neighborhood of $x$ is
    isometric to a cyclic $n$-cover of the punctured disc.
  \item[Infinite angle conical singularity] We say that $x$ is an
    \emph{infinite angle conical singularity} if it has 
    a punctured neighborhood of $x$ is isometric to an infinite cyclic
    cover of the punctured disc.
  \item[Wild singularity] In all other situations we say $x$ is a
    \emph{wild singularity}.
  \end{description}
\end{definition}

\begin{remark}
  We note that a wild singularity may have rotational components of
  finite cone angle that do not correspond to metric cones embedded in
  the surface.
\end{remark}

A classic example of a removable singularity is the singularity
occurring in a torus constructed from four copies of the unit square
properly identified.  

An example of a finite angle conical singularity
is illustrated in Figure \ref{fig:t-shapedPolygonFlatSurface}. The
cone angle of the singularity arising from the corners of $T_0$
measuring $3\pi/2$ (when measured from within the T-shaped polygon)
measures $6\pi$.  Such a singularity thus constitutes a finite angle
conical singularity of the corresponding translation surface.
(Corners of the T-fractal billiard table with angle $\frac{\pi}{2}$
give rise to removable singularities of the associated translation
surface.)  An
example of an infinite angle conical singularity is seen in the
infinite staircase surfaces studied 
in 
\cite{HubertSchmithuesen}, \cite{HooperWeiss}, \cite{ConzeGutkin},
\cite{HooperHubertWeiss}, and \cite{FraczekUlcigrai}.  

 The translation surface illustrated in
Figure \ref{fig:chamanara} gives an example of a wild singularity and
is due to Chamanara \cite{Chamanara}.   Consider taking a unit square and
cutting each edge into pieces of length $\sfrac{1}{2}$,
$\sfrac{1}{4}$, $\sfrac{1}{8}$, $\ldots$ with parallel edges cut in
opposite orders, and then parallel edges of the same length identified
as indicated in Figure~\ref{fig:chamanara}.  The metric completion of
this surface adds a single point corresponding to the corners of the
square and the endpoints of the cuts, and this is a wild singularity:
no neighborhood of this point can be isometric to a cover of the
punctured disc as there are arbitrarily short geodesic loops based at
the singularity.

  \begin{figure}
    \begin{tikzpicture}[scale=2]
      \draw (-1, -1) rectangle (1, 1);
      \foreach [count=\x] \n in {A, B, ..., H} {
        \foreach \y in {1, -1} {
          \begin{scope}[yscale=\y,xscale=\y]
            \filldraw ($(1-2*1/2^\x, 1)$) circle (0.005in);
          \end{scope}
        }
      }
      \foreach [count=\x] \n in {A, B, ..., D} {
        \foreach \y in {1, -1} {
          \begin{scope}[yscale=\y,xscale=\y]
            \node at ($(1-3/2^\x, 1.1)$) {\n};
          \end{scope}
        }
      }
      \foreach [count=\x] \n in {a, b, ..., h} {
        \foreach \y in {1, -1} {
          \begin{scope}[yscale=\y,xscale=\y]
            \filldraw ($(1, 1-2*1/2^\x)$) circle (0.005in);
          \end{scope}
        }
      }
      \foreach [count=\x] \n in {a, b, ..., d} {
        \foreach \y in {1, -1} {
          \begin{scope}[yscale=\y,xscale=\y]
            \node at ($(1.1, 1-3/2^\x)$) {\n};
          \end{scope}
        }
      }
    \end{tikzpicture}
    \caption{The Chamanara surface. Parallel edges of the same length
      are identified.  The first few horizontal and vertical
      identifications are labeled in this diagram.}
    \label{fig:chamanara}
  \end{figure}

\begin{remark}
In \cite{BowmanValdez}, the set of wild singularities is assumed to be
discrete to avoid trivial examples of wild singularities such as the
boundary points on the unit disc. In this paper we remove this
restriction, because, as we will see, a non-discrete set of points
naturally arising from the metric completion of the unfolding of a
billiard table forms a set of singularities with infinitely-many
rotational components. The authors believe such singularities are
still worthy of being called \emph{wild singularities}, as will be discussed in Section \ref{sec:elusiveSingularitiesWildSingularities}.

	
\end{remark}

\section{An explicit construction of the $T$-fractal translation surface}
\label{sec:constructionTheT-FractalFlatSurface}

\subsection{The $T$-fractal billiard table and $T$-fractal translation surface}
\label{subsec:TfractalBilliardTfractalSurface}

Though $\tfractalr$ is not strictly a polygonal region, the
``unfolding procedure'' described by \cite{FoxKershner} and
\cite{KatokZemlyakov} may still be applied to obtain a translation
surface whose geodesics project to billiard trajectories in
$\tfractalr$.  In particular, $\tfractal$ and $\tfractalr$ consist
only of horizontal and vertical edges,\footnote{$\tfractal$ is
  $\tfractalr$ with corners from each $T_n$ included.}  so the group
generated by linear reflections in its sides is the group generated by
two orthogonal reflections, the Klein four-group,
$\mathbb{Z}_2 \oplus \mathbb{Z}_2$.

\begin{definition}[The $T$-fractal translation surface]
As before, let $\tfractalr$ denote the fractal billiard $\tfractal$ with
corners removed.  \label{ntn:tsurfr} Then $\tsurfr$ denotes the fractal translation
surface and is constructed from four copies of $\tfractalr$
by properly identifying horizontal and vertical sides as shown in
Figure \ref{fig:unfolded_identifications}.

\end{definition}

  \begin{figure}
    \centering
    \begin{tikzpicture}[scale=0.75]
      \begin{scope}[xshift=\xs,yshift=\ys]
        \fractalT{\iterates}
        \node[below] at (2, 0) {$a$};
        \node[right] at (3, 1) {$b$};
        \node[left] at (1, 1) {$c$};
        \node[below] at (3.5, 2) {$d$};
        \node[below] at (0.5, 2) {$e$};
      \end{scope}

      \begin{scope}[xshift=-\xs,yshift=\ys,xscale=-1]
        \fractalT{\iterates}
        \node[below] at (2, 0) {$h$};
        \node[left] at (3, 1) {$b$};
        \node[right] at (1, 1) {$c$};
        \node[below] at (3.5, 2) {$f$};
        \node[below] at (0.5, 2) {$g$};
      \end{scope}

      \begin{scope}[xshift=-\xs,yshift=-\ys,yscale=-1,xscale=-1]
        \fractalT{\iterates}
        \node[above] at (2, 0) {$h$};
        \node[above] at (3.5, 2) {$f$};
        \node[above] at (0.5, 2) {$g$};
        \node[left] at (3, 1) {$i$};
        \node[right] at (1, 1) {$j$};
      \end{scope}

      \begin{scope}[xshift=\xs,yshift=-\ys,yscale=-1]
        \fractalT{\iterates}
        \node[above] at (2, 0) {$a$};
        \node[above] at (3.5, 2) {$d$};
        \node[above] at (0.5, 2) {$e$};
        \node[right] at (3, 1) {$i$};
        \node[left] at (1, 1) {$j$};
      \end{scope}
    \end{tikzpicture}
    \caption{The surface $\tsurfr$.  Only the first few
      edges in the copies of the billiard table $\tfractal$ are labeled
      to indicate how sides are glued together to form $\tsurfr$.}
    \label{fig:unfolded_identifications}
  \end{figure}

The surface $\tsurfr$ is not a complete metric space, and so we
consider its metric completion which we will denote
by \label{ntn:tsurf} $\tsurf$, where the metric $d: \tsurf\times
\tsurf \to \mathbb{R}$ is the completion of the Riemannian metric
$|dz|^2 = dx^2 + dy^2$ from $\mathbb{R}^2$ pulled back to $\tsurfr$
using local charts.  A
first observation about $\tsurf$ is that it contains infinitely-many
cone points of cone angle $6\pi$ coming from the corners of $T_n$
where the scaled $2 \times \sfrac{1}{2}$ rectangle is placed on top of
the scaled $1 \times 1$ square.  The singular points in $\tsurf
\setminus \tsurfr$ which are more interesting, however, are those
points which come from the ``top'' of the $T$-fractal.  In previous
joint works of the second author, M. L. Lapidus and R. L. Miller,
these were called the \emph{elusive points} of the $T$-fractal
billiard,\footnote{The elusive points of the $T$-fractal billiard, in
our new notation, are $\overline{\tfractal}\setminus \tfractal$.} and
so we will refer to these as \emph{elusive singularities} of the
$T$-fractal translation surface $\tsurfr$.  (A precise definition of
elusive singularities is given below.)

\begin{remark}

 In previous articles (e.g., \cite{LapNie1} and \cite{LapMilNie}), the set of elusive points of the $T$-fractal billiard was a subset of the
  $T$-fractal billiard and constituted a connected interval.  Such
  points are now viewed as elements in  $\overline{\tfractal}\setminus \tfractal$.
  What is shown in this article is that the set of elusive
  singularities of the $T$-fractal translation surface (the fractal analog of a translation surface) is a totally disconnected
  set and is only present in the metric completion of the $T$-fractal
  surface.  Moreover, the metric completion of the $T$-fractal surface is in
  fact not a surface in the mathematical sense of the word `surface'.  
  
 \end{remark}
  
  Since $\tsurf$ is the metric completion of $\tsurfr$, it follows that $\tsurf$ is not simply four copies of $\overline{\tfractal}$ properly identified, as this would imply that the elusive singularities fo the T-fractal translation surface form a connected set, contrary to what we show later.  We argue that viewing $\tsurf$ as four copies of $\overline{\tfractal}$ identified by horizontal and vertical reflections and translations is incorrect, this being the proposed construction in previous articles by the second author.\footnote{In \cite{LapMilNie}, it was proposed that four copies of $\tsurfr$ could be properly identified to construct a reasonable notion of a $T$-fractal translation surface.  As will be shown, the appropriate fractal  analog of a translation surface corresponding to the $T$-fractal billiard is constructed using scaled copies of a particular surface we will call the quad-$T$ surface $\mathscr{Q}$.}

In order to study the elusive singularities of $\tsurfr$ we will
consider a special class of compact translation surfaces with boundary which are embedded in $\tsurf$ and which we call \emph{quad-T subsurfaces}.

\subsection{The quad-T subsurfaces}
Just as we imagine the $T$-fractal billiard $\tfractal$ as being built
from scaled copies of the original T-shaped polygon, we may imagine $\tsurfr$ as being built from scaled copies of a certain
translation surface with boundary indicated in Figure~\ref{fig:quadt}.
We call the surface displayed in  Figure~\ref{fig:quadt} a \emph{quad-T
  surface}.  When corners are absent from the quad-T surface, we denote the surface in Figure~\ref{fig:quadt} by \label{ntn:Qr} $\Qr$ and when corners are present, the notation used is \label{ntn:Q} $\Q$.  Then, $\tsurf$ is obtained by computing the metric completion of appropriately scaled copies of $\Q$ glued together at their boundary components.  To distinguish the different copies of the quad-T surface in $\tsurf$, we index the copies by binary strings.

Let \label{ntn:scriptB} $\mathscr{B} = \{0, 1\}$, and let \label{ntn:scriptBstar} $\mathscr{B}^*$ be the set of all finite binary strings with \label{ntn:epsilon} $\epsilon$ denoting the empty string.  For each string $s \in \mathscr{B}^*$, let \label{ntn:Qrs} $\Qr^s$ \label{ntn:Qs} ($\Q^s$) denote a copy of $\Qr$ ($\Q$) scaled by $2^{-|s|}$ where $|s|$ is the length of the string $s$. Labeling the boundary components of $\Q$ as indicated in Figure~\ref{fig:quadt}, we let \label{ntn:gammais} \label{ntn:sigmais} $\gamma_1^s, \ldots, \gamma_6^s, \sigma_1^s, \ldots, \sigma_6^s$ denote the corresponding boundary components of $\Q^s$.  We note that when corners are removed from $\Q^s$ to produce $\Qr^s$, endpoints of boundary components are also removed, this being important when constructing a \textit{fractal interval exchange transformation} in Section \ref{sec:IETs} and Subsection \ref{subsec:FIET_irrational_direction}.

\def\gammaColor{dashed}
\def\sigmaColor{dashed}

\begin{figure}[h!]
  \centering
  \begin{tikzpicture}[scale=0.66]
    \begin{scope}[thick,xshift=1.25in]
      \draw (0, 0) -- node[scale=0.66,left] {$E$} (0, 2) 
                   -- node[scale=0.66,below] {$F$} (-1, 2)
                   -- node[scale=0.66,left] {$G$} (-1, 3);
      \draw (0, 3) -- node[scale=0.66,above] {$A$} (2, 3);
      \draw (3, 3) -- node[scale=0.66,right] {$B$} (3, 2)
                   -- node[scale=0.66,below] {$C$} (2, 2)
                   -- node[scale=0.66,right] {$D$} (2, 0);

      \draw[\sigmaColor] (-1, 3) -- node[above] {$\sigma_4$} (0, 3);
      \draw[\sigmaColor] (2, 3) -- node[above] {$\sigma_6$} (3, 3);
      \draw[\gammaColor] (2, 0) -- node[below] {$\gamma_5$} (0, 0);
    \end{scope}

    \begin{scope}[thick,xshift=-1.25in]
      \draw (0, 0) -- node[scale=0.66,left] {$D$} (0, 2)
                   -- node[scale=0.66,below] {$J$} (-1, 2)
                   -- node[scale=0.66,left] {$B$} (-1, 3);
      \draw (0, 3) -- node[scale=0.66,above] {$H$} (2, 3);
      \draw (3, 3) -- node[scale=0.66,right] {$G$} (3, 2)
                   -- node[scale=0.66,below] {$I$} (2, 2)
                   -- node[scale=0.66,right] {$E$} (2, 0);

      \draw[\sigmaColor] (-1, 3) -- node[above] {$\sigma_1$}(0, 3);
      \draw[\sigmaColor] (2, 3) -- node[above] {$\sigma_3$} (3, 3);
      \draw[\gammaColor] (2, 0) -- node[below] {$\gamma_2$} (0, 0);
    \end{scope}

    \begin{scope}[thick,xshift=-1.25in,yscale=-1, yshift=0.5in]
      \draw (0, 0) -- node[scale=0.66,left] {$N$} (0, 2)
                   -- node[scale=0.66,above] {$J$} (-1, 2)
                   -- node[scale=0.66,left] {$M$} (-1, 3);
      \draw (0, 3) -- node[scale=0.66,below] {$H$} (2, 3);
      \draw (3, 3) -- node[scale=0.66,right] {$L$} (3, 2)
                   -- node[scale=0.66,above] {$I$} (2, 2)
                   -- node[scale=0.66,right] {$K$} (2, 0);

      \draw[\gammaColor] (-1, 3) -- node[below] {$\gamma_1$} (0, 3);
      \draw[\gammaColor] (2, 3) -- node[below] {$\gamma_3$} (3, 3);
      \draw[\sigmaColor] (2, 0) -- node[above] {$\sigma_2$} (0, 0);
    \end{scope}
    
    \begin{scope}[thick,xshift=1.25in,yscale=-1, yshift=0.5in]
      \draw (0, 0) -- node[scale=0.66,left] {$K$} (0, 2)
                   -- node[scale=0.66,above] {$F$} (-1, 2)
                   -- node[scale=0.66,left] {$L$} (-1, 3);
      \draw (0, 3) -- node[scale=0.66,below] {$A$} (2, 3);
      \draw (3, 3) -- node[scale=0.66,right] {$M$} (3, 2)
                   -- node[scale=0.66,above] {$C$} (2, 2)
                   -- node[scale=0.66,right] {$N$} (2, 0);

      \draw[\gammaColor] (-1, 3) -- node[below] {$\gamma_4$} (0, 3);
      \draw[\gammaColor] (2, 3) -- node[below] {$\gamma_6$} (3, 3);
      \draw[\sigmaColor] (2, 0) -- node[above] {$\sigma_5$} (0, 0);
    \end{scope}
  \end{tikzpicture}
  \caption{The quad-T surface, $\Q$.  Dashed line segments are
    boundary components, and solid line segments with the same label
    are identified by translation as indicated by the letters $A-M$ so
    that the orientation of the flow is preserved.}
  \label{fig:quadt}
\end{figure}

We may build up to  $\tsurf$ by gluing together particular
quad-T surfaces at their boundary components.  More to the point, if
for some $i_1$ and $i_2$, $\gamma_{i_1}^s$ and $\sigma_{i_2}^s$ are
boundary components of $\Qr^s$, then $\gamma_{i_1}^s$ and
$\sigma_{i_2}^s$ are glued to $\sigma_{j_1}^{s'}$ and
$\gamma_{j_2}^{s'}$, respectively, where $|s'| = |s|+1$. Such a
construction preserves the orientation of the flow and, consequently,
gives us a translation surface.  The initial stages of the
construction of $\tsurf$ (or $\tsurfr$ when corners are omitted) are shown in Figure \ref{fig:QepsilongQ0Q1}.

\def\gammaColor{dashed}
\def\sigmaColor{dashed}

\begin{figure}[h!]
  \centering
  \begin{tikzpicture}[scale=0.66]
    \begin{scope}[thick,xshift=1.25in]
      \draw (0, 0) -- node[scale=0.66,left] {$E$} (0, 2) 
                   -- node[scale=0.66,below] {$F$} (-1, 2)
                   -- node[scale=0.66,left] {$G$} (-1, 3);
      \draw (0, 3) -- node[scale=0.66,above] {$A$} (2, 3);
      \draw (3, 3) -- node[scale=0.66,right] {$B$} (3, 2)
                   -- node[scale=0.66,below] {$C$} (2, 2)
                   -- node[scale=0.66,right] {$D$} (2, 0);

      \draw[\sigmaColor] (-1, 3) -- node[above] {$\sigma_4$} (0, 3);
      \draw[\sigmaColor] (2, 3) -- node[above] {$\sigma_6$} (3, 3);
      \draw[\gammaColor] (2, 0) -- node[below] {$\gamma_5$} (0, 0);
    \end{scope}
\begin{scope}[scale=.5,thick,xshift=1.7in,yshift=3.5in]
      \draw (0, 0) --  (0, 2) 
                   --  (-1, 2)
                   -- (-1, 3);
      \draw (0, 3) -- (2, 3);
      \draw (3, 3) --  (3, 2)
                   --  (2, 2)
                   -- (2, 0);

      \draw[\sigmaColor] (-1, 3) -- node[above] {$\sigma_4^1$} (0, 3);
      \draw[\sigmaColor] (2, 3) -- node[above] {$\sigma_6^1$} (3, 3);
      \draw[\gammaColor] (2, 0) -- node[below] {$\gamma_5^1$} (0, 0);
    \end{scope}

\begin{scope}[scale=.5,thick,xshift=4in,yshift=3.5in]
      \draw (0, 0) --  (0, 2) 
                   --  (-1, 2)
                   -- (-1, 3);
      \draw (0, 3) -- (2, 3);
      \draw (3, 3) --  (3, 2)
                   --  (2, 2)
                   -- (2, 0);

      \draw[\sigmaColor] (-1, 3) -- node[above] {$\sigma_4^0$} (0, 3);
      \draw[\sigmaColor] (2, 3) -- node[above] {$\sigma_6^0$} (3, 3);
      \draw[\gammaColor] (2, 0) -- node[below] {$\gamma_5^0$} (0, 0);
    \end{scope}

    \begin{scope}[thick,xshift=-1.25in]
      \draw (0, 0) -- node[scale=0.66,left] {$D$} (0, 2)
                   -- node[scale=0.66,below] {$J$} (-1, 2)
                   -- node[scale=0.66,left] {$B$} (-1, 3);
      \draw (0, 3) -- node[scale=0.66,above] {$H$} (2, 3);
      \draw (3, 3) -- node[scale=0.66,right] {$G$} (3, 2)
                   -- node[scale=0.66,below] {$I$} (2, 2)
                   -- node[scale=0.66,right] {$E$} (2, 0);

      \draw[\sigmaColor] (-1, 3) -- node[above] {$\sigma_1$}(0, 3);
      \draw[\sigmaColor] (2, 3) -- node[above] {$\sigma_3$} (3, 3);
      \draw[\gammaColor] (2, 0) -- node[below] {$\gamma_2$} (0, 0);
    \end{scope}

\begin{scope}[scale=.5,thick,xshift=-.95in,yshift=3.5in]
      \draw (0, 0) --  (0, 2) 
                   --  (-1, 2)
                   -- (-1, 3);
      \draw (0, 3) -- (2, 3);
      \draw (3, 3) --  (3, 2)
                   --  (2, 2)
                   -- (2, 0);

      \draw[\sigmaColor] (-1, 3) -- node[above] {$\sigma_1^1$} (0, 3);
      \draw[\sigmaColor] (2, 3) -- node[above] {$\sigma_3^1$} (3, 3);
      \draw[\gammaColor] (2, 0) -- node[below] {$\gamma_2^1$} (0, 0);
    \end{scope}

\begin{scope}[scale=.5,thick,xshift=-3.3in,yshift=3.5in]
      \draw (0, 0) --  (0, 2) 
                   --  (-1, 2)
                   -- (-1, 3);
      \draw (0, 3) -- (2, 3);
      \draw (3, 3) --  (3, 2)
                   --  (2, 2)
                   -- (2, 0);

      \draw[\sigmaColor] (-1, 3) -- node[above] {$\sigma_1^0$} (0, 3);
      \draw[\sigmaColor] (2, 3) -- node[above] {$\sigma_3^0$} (3, 3);
      \draw[\gammaColor] (2, 0) -- node[below] {$\gamma_2^0$} (0, 0);
    \end{scope}

    \begin{scope}[thick,xshift=-1.25in,yscale=-1, yshift=0.5in]
      \draw (0, 0) -- node[scale=0.66,left] {$N$} (0, 2)
                   -- node[scale=0.66,above] {$J$} (-1, 2)
                   -- node[scale=0.66,left] {$M$} (-1, 3);
      \draw (0, 3) -- node[scale=0.66,below] {$H$} (2, 3);
      \draw (3, 3) -- node[scale=0.66,right] {$L$} (3, 2)
                   -- node[scale=0.66,above] {$I$} (2, 2)
                   -- node[scale=0.66,right] {$K$} (2, 0);

      \draw[\gammaColor] (-1, 3) -- node[below] {$\gamma_1$} (0, 3);
      \draw[\gammaColor] (2, 3) -- node[below] {$\gamma_3$} (3, 3);
      \draw[\sigmaColor] (2, 0) -- node[above] {$\sigma_2$} (0, 0);
    \end{scope}
    
     \begin{scope}[scale=.5, thick,xshift=-.95in,yscale=-1, yshift=4.55in]
      \draw (0, 0) -- (0, 2)
                   --  (-1, 2)
                   --  (-1, 3);
      \draw (0, 3) --  (2, 3);
      \draw (3, 3) --  (3, 2)
                   --  (2, 2)
                   --  (2, 0);

      \draw[\gammaColor] (-1, 3) -- node[below] {$\gamma_1^1$} (0, 3);
      \draw[\gammaColor] (2, 3) -- node[below] {$\gamma_3^1$} (3, 3);
      \draw[\sigmaColor] (2, 0) -- node[above] {$\sigma_2^1$} (0, 0);
    \end{scope}

     \begin{scope}[scale=.5, thick,xshift=-3.3in,yscale=-1, yshift=4.55in]
      \draw (0, 0) --  (0, 2)
                   --  (-1, 2)
                   -- (-1, 3);
      \draw (0, 3) --  (2, 3);
      \draw (3, 3) --  (3, 2)
                   -- (2, 2)
                   -- (2, 0);

      \draw[\gammaColor] (-1, 3) -- node[below] {$\gamma_1^0$} (0, 3);
      \draw[\gammaColor] (2, 3) -- node[below] {$\gamma_3^0$} (3, 3);
      \draw[\sigmaColor] (2, 0) -- node[above] {$\sigma_2^0$} (0, 0);
    \end{scope}

    \begin{scope}[thick,xshift=1.25in,yscale=-1, yshift=0.5in]
      \draw (0, 0) -- node[scale=0.66,left] {$K$} (0, 2)
                   -- node[scale=0.66,above] {$F$} (-1, 2)
                   -- node[scale=0.66,left] {$L$} (-1, 3);
      \draw (0, 3) -- node[scale=0.66,below] {$A$} (2, 3);
      \draw (3, 3) -- node[scale=0.66,right] {$M$} (3, 2)
                   -- node[scale=0.66,above] {$C$} (2, 2)
                   -- node[scale=0.66,right] {$N$} (2, 0);

      \draw[\gammaColor] (-1, 3) -- node[below] {$\gamma_4$} (0, 3);
      \draw[\gammaColor] (2, 3) -- node[below] {$\gamma_6$} (3, 3);
      \draw[\sigmaColor] (2, 0) -- node[above] {$\sigma_5$} (0, 0);
    \end{scope}
    
    \begin{scope}[scale=.5,thick,xshift=1.75in,yscale=-1, yshift=4.55 in]
      \draw (0, 0) -- (0, 2)
                   --  (-1, 2)
                   --  (-1, 3);
      \draw (0, 3) --  (2, 3);
      \draw (3, 3) -- (3, 2)
                   --  (2, 2)
                   --  (2, 0);

      \draw[\gammaColor] (-1, 3) -- node[below] {$\gamma_4^1$} (0, 3);
      \draw[\gammaColor] (2, 3) -- node[below] {$\gamma_6^1$} (3, 3);
      \draw[\sigmaColor] (2, 0) -- node[above] {$\sigma_5^1$} (0, 0);
    \end{scope}

    \begin{scope}[scale=.5,thick,xshift=4in,yscale=-1, yshift=4.55in]
      \draw (0, 0) -- (0, 2)
                   --  (-1, 2)
                   --  (-1, 3);
      \draw (0, 3) --  (2, 3);
      \draw (3, 3) -- (3, 2)
                   --  (2, 2)
                   --  (2, 0);

      \draw[\gammaColor] (-1, 3) -- node[below] {$\gamma_4^0$} (0, 3);
      \draw[\gammaColor] (2, 3) -- node[below] {$\gamma_6^0$} (3, 3);
      \draw[\sigmaColor] (2, 0) -- node[above] {$\sigma_5^0$} (0, 0);
    \end{scope}
  \end{tikzpicture}
  \caption{We show here the quad-T subsurfaces $\Q^\epsilon$, $\Q^0$ and $\Q^1$ as they would appear when embedded in $\tsurf$ (or, how $\Qr^\epsilon$, $\Qr^0$ and $\Qr^1$ would appear when embedded in $\tsurfr$). We also label the boundary components of each scaled copy of $\Q$ so that the reader can better visualize the gluings that give rise to the quad-T construction of $\tsurf$ (or, $\tsurfr$).}
  \label{fig:QepsilongQ0Q1}
\end{figure}

Precisely,
we identify $\gamma_2^\epsilon \sim \sigma_2^\epsilon$ and
$\gamma_5^\epsilon \sim \sigma_5^\epsilon$, and then for each nonempty string
$s$ we perform the following identifications:
\[
\begin{array}{cc}
  \sigma_1^s \sim \gamma_2^{s0}, &
  \sigma_6^s \sim \gamma_5^{s0}, \\
  \gamma_1^s \sim \sigma_2^{s0}, &
  \gamma_6^s \sim \sigma_5^{s0}, \\
  \sigma_3^s \sim \gamma_2^{s1}, &
  \sigma_4^s \sim \gamma_5^{s1}, \\
  \gamma_3^s \sim \sigma_2^{s1}, &
  \gamma_4^s \sim \sigma_5^{s1}.
\end{array}
\]
All identifications are translations between parallel line segments of
equal length, and the above identifications give the surface
$\left(\bigcup_{s\in \mathscr{B}^*}{\Q^s}\right)/\sim$.  Such a surface is equal to four copies of $T_\infty$ properly identified (i.e., by the the action of $Z_2\oplus Z_2$ on $T_\infty$).  As previously mentioned, however, the metric completion of $T_\infty$ acted on by the Klein group is not the metric completion of $\bigcup_{s\in \mathscr{B}^*} Q^s/\sim$.   We may  think of each $\Q^s$ as being
embedded in $\tsurf$ and $\Qr^s$ being embedded in $\tsurfr$; see Figure~\ref{fig:labels}.  We close this section with a useful definition and notation that will be used in the subsequent sections.


\begin{figure}
  \centering
  \begin{tikzpicture}[scale=0.75]
    \begin{scope}[xshift=\xs,yshift=\ys]
      \fractalT{\iterates}
      \labeledT{\iterates}
    \end{scope}

    \begin{scope}[xshift=-\xs,yshift=\ys,xscale=-1]
      \fractalT{\iterates}
      \labeledT{\iterates}
    \end{scope}

    \begin{scope}[xshift=-\xs,yshift=-\ys,yscale=-1,xscale=-1]
      \fractalT{\iterates}
      \labeledT{\iterates}
    \end{scope}

    \begin{scope}[xshift=\xs,yshift=-\ys,yscale=-1]
      \fractalT{\iterates}
      \labeledT{\iterates}
    \end{scope}
  \end{tikzpicture}
  \caption{The scaled quad-T surfaces which appear in $\tsurf$.}
  \label{fig:labels}
\end{figure} 

\begin{definition}[Branch of $\tsurf$ rooted at $\Q^s$]
  For each $s\in \mathscr{B}^*$ we define a \emph{branch of $\tsurf$
    rooted at $\Q^s$}, denoted \label{ntn:branch} $\mathcal{B}^s$, to be the union of all
  quad-T subsurfaces whose indexing string contains $s$ as a
  substring,
  \[
    \mathcal{B}^s = \bigcup_{t \in \mathscr{B}^*} \Q^{st}/\sim.
  \]

\noindent For notational purposes, 
\[
\mathring{\mathcal{B}}^s = \bigcup_{t\in\mathscr{B}^*} \Qr^{st}/\sim,
\]
\noindent and $\mathring{\mathcal{B}}^s$ is referred to as a branch of $\tsurfr$ rooted at $\Qr^s$.  
\end{definition}

As expected, $\mathcal{B}^\epsilon = \bigcup_{t\in\mathscr{B}^*} \Q^{t}/\sim$.  We now define a \textit{truncated branch of $\tsurf$} as follows.

\begin{definition}
	For each $n\in\N$ and $s\in\mathscr{B}^*$, we define a \textit{truncated branch of $\tsurf$ rooted at $\Q^s$}, denoted by $\mathcal{B}^s_n$, as the following finite union of quad-T subsurfaces:
	
	\[
		\mathcal{B}^s_n = \bigcup_{t\in\mathscr{B}^*, |t|\leq n} \Q^{st}/\sim.
	\]
\end{definition}

While $\mathcal{B}^s$ and $\mathring{\mathcal{B}}^s$ are surfaces
without boundary, a truncated branch of T does have boundary, since
$\gamma_i^{st}$, $i= 1, 3, 4, 6$, and $\sigma_i^{st}$,
$i= 1, 3, 4, 6$, of $\Q^{st}$, $|t| = n$ will not be glued to any
other segments, this being a key difference between a branch and a
truncated branch.

\subsection{Elusive singularities}
In order to understand the points added by the metric completion of
$\tsurfr$, we will show that equivalence classes of Cauchy sequences
on $\tsurfr$ which do not converge to a point of
$\mathcal{B}^\epsilon$ may be thought of as increasing sequences of
quad-T subsurfaces.

\begin{definition}
  We define an \emph{elusive
    singularity}\label{ntn:elusiveSingularity} of $\tsurfr$ to be a
  point of $\tsurf$ which is not contained in any quad-T subsurface
  $\Q^s$ and denote the set of elusive singularities $\esurf$ as
  \[
    \esurf \coloneqq \tsurf \setminus \mathcal{B}^\epsilon.
  \]
\end{definition}

Intuitively, elusive singularities of $\tsurfr$ correspond to points
in (the unfolding of) the $T$-fractal billiard table, $\tfractal$, but
which do not appear in (the unfolding of) any finite approximation,
$T_n$.  This analogy does not hold rigorously, as the set of elusive
points will be shown to be a Cantor set in Theorem
\ref{prop:elusive_cantor}.

\begin{definition}[Address of an elusive singularity]
  \label{def:addressOfElusivePoint}
  Consider $x\in \esurf=\tsurf\setminus\mathcal{B}^\epsilon$ and
  $(x_n)_{n\in\N}$ the Cauchy sequence in $\tsurfr$ converging to $x$
  with the property that $|\sigma(x_n)|$ is a strictly increasing
  sequence.  The \textit{address of $x$}, denoted by \label{ntn:alphax}
  $\alpha(x)$ is then given by the bits $b_i$, where for each $i$,
  $b_{i+1}$ is the bit appended to $\sigma(x_i)$ to produce
  $\sigma(x_{i+1})$, meaning that
  $\sigma(x_{i+1}) = \sigma(x_i)b_{i+1}$.
\end{definition}

In order to understand the geodesic rays which approach elusive
singularities, it will be convenient to study some simple dynamical
properties of a special type of \emph{interval exchange
  transformation} associated with the surface.

\section{Interval exchange transformations}
\label{sec:IETs}
In this section we briefly discuss specific special interval exchange
transformations on the T-fractal surface.  We will use the
transformations discussed in this section in Section
\ref{sec:linearApproachesExist} to show that for each address of an
elusive singularity there exists a linear approach to that
singularity.

An \emph{interval exchange transformation}, or IET, is simply a map
from an interval to itself which is an injective piecewise
translation.  That is, the interval is partitioned into subintervals
which are then rearranged by translations in such a way that the
images of the subintervals again give a partition of the original
interval.  IETs naturally arise in the study of translation surfaces
as Poincar\'e sections of geodesic flows.  That is, fixing any
direction on the surface and any transverse geodesic interval, the
first return map to the interval by the geodesic flow in the fixed
direction is an interval exchange.

Typically, IETs are the first return map from a single connected
interval to itself, but we may also consider the first return map to a
collection of disjoint intervals.  This will still be an interval
exchange, possibly moving subintervals between these disjoint
intervals, if the intervals are all parallel.

In our setting, we will consider intervals which come from the
bottom-most horizontal edges of a quad-T subsurface.  In
Figure~\ref{fig:labels}, these are the dashed intervals together with
horizontal intervals in the $\mathcal{Q}^\epsilon$ subsurface which
appear at the bottom of the two T's in the top half of the picture,
and the intervals at the top of the two T's in the bottom half of the
picture.

In order to study this interval exchange we make use of the
self-similarity of the $T$-fractal.\footnote{To be precise, the
  $T$-fractal is not, strictly speaking, self-similar in the sense
  that it is the unique fixed point attractor of an iterated function
  system.  Rather, the base $T$ shape is repeated at smaller and
  smaller scales with increasing frequency, giving the feeling of a
  fractal.}  In particular, we note that since the way points move
from the interval at the base of $\mathcal{Q}^s$ is the same as the
way they move from points on any other $\mathcal{Q}^t$.  To be more
precise, we consider the interval exchange defined on the quad-T
subsurfaces as follows.


We partition the boundaries of the quad-T $\mathcal{Q}$ into two
halves which we call \label{ntn:Gamma} $\Gamma$ and \label{ntn:Sigma}
$\Sigma$, with $\Gamma = \{\gamma_1, \ldots, \gamma_6\}$ and
$\Sigma = \{\sigma_1,\ldots , \sigma_6\}$ using the subintervals
indicated in Figure~\ref{fig:quadt}.  For each direction
$\theta \in (0, \pi)$ we consider the map \label{ntn:FthetaFIET}
$F_\theta : \Gamma \to \Sigma$ defined by following the geodesic ray
emitted from $x \in \Gamma$ in direction $\theta$ until reaching
$F_\theta(x) \in \Sigma$.  Noting that $\Gamma$ and $\Sigma$ each
consists of six intervals of the same sizes (i.e.,
$|\gamma_i| = |\sigma_i|$), we may interpret $F_\theta$ as an interval
exchange on one interval $I$.  Assigning coordinates to $\mathcal{Q}$
in the natural way so that the left-most points (the corners of
$\gamma_1$ and $\sigma_1$) have $x$-coordinate $0$, we identify $I$
with the interval $I = [0, 4]$.  It is clear that the interval
exchange on $I$ can be computed by simply following the geodesics
emitted in directions $\theta$, $\pi+\theta$ from the conical
singularities to $\Gamma \sqcup \Sigma$ to first determine the
discontinuities of $F_\theta$, and the subintervals of $\Gamma$
between two such rays are mapped by translation to $\Sigma$.  See
Figure~\ref{fig:computingIET} for an example in the case of direction
$\theta = \frac{\pi}{4}$.

\begin{figure}
  \centering
  \begin{tikzpicture}
    \begin{scope}[thick,xshift=1in]
      \filldraw[color=red,opacity=0.2] (0, 1) -- (0, 2) -- (1, 3) --
      (2, 3) -- cycle;
      \filldraw[color=red,opacity=0.2] (2, 2) -- (3, 3) -- (3, 2) --
      cycle;

      \filldraw[color=blue,opacity=0.2] (0, 0) -- (2, 2) -- (2, 1) --
      (1, 0) -- cycle;
      \filldraw[color=blue,opacity=0.2] (-1, 2) -- (-1, 3) -- (0, 3)
      -- cycle;

      \filldraw[color=green,opacity=0.2] (1, 0) -- (2, 1) -- (2, 0) --
      cycle;

      \filldraw[color=orange,opacity=0.2] (0, 0) -- (3, 3) -- (2, 3) --
      (0, 1) -- cycle;

      \filldraw[color=green,opacity=0.5] (-1, 2) -- (0, 3) -- (1, 3)
      -- (0, 2) -- cycle;

      \draw (0, 0) -- (0, 2) -- (-1, 2) -- (-1, 3);
      \draw[color=red] (-1, 3) -- node[above] {$\sigma_4$} (0, 3);
      \draw (0, 3) -- (2, 3);
      \draw[color=red] (2, 3) -- node[above] {$\sigma_6$} (3, 3);
      \draw (3, 3) -- (3, 2) -- (2, 2) -- (2, 0);
      \draw[color=green!80!black] (2, 0) -- node[below] {$\gamma_5$}
      (0, 0);
    \end{scope}
    \begin{scope}[thick,xshift=-1in]
      \filldraw[color=blue,opacity=0.2] (0, 1) -- (0, 2) -- (1, 3) --
      (2, 3) -- cycle;
      \filldraw[color=blue,opacity=0.2] (2, 2) -- (3, 3) -- (3, 2) --
      cycle;

      \filldraw[color=red,opacity=0.2] (0, 0) -- (2, 2) -- (2, 1) --
      (1, 0) -- cycle;
      \filldraw[color=red,opacity=0.2] (-1, 2) -- (-1, 3) -- (0, 3)
      -- cycle;

      \filldraw[color=green,opacity=0.2] (0, 0) -- (3, 3) -- (2, 3) --
      (0, 1) -- cycle;
      
      \filldraw[color=orange,opacity=0.2] (1, 0) -- (2, 1) -- (2, 0) --
      cycle;

      \filldraw[color=orange,opacity=0.5] (-1, 2) -- (0, 3) -- (1, 3)
      -- (0, 2) -- cycle;

      \draw (0, 0) -- (0, 2) -- (-1, 2) -- (-1, 3);
      \draw[color=red] (-1, 3) -- node[above] {$\sigma_1$}(0, 3);
      \draw (0, 3) -- (2, 3);
      \draw[color=red] (2, 3) -- node[above] {$\sigma_3$} (3, 3);
      \draw (3, 3) -- (3, 2) -- (2, 2) -- (2, 0);
      \draw[color=green!80!black] (2, 0) -- node[below] {$\gamma_2$}
      (0, 0);
    \end{scope}
    \begin{scope}[thick,xshift=-1in,yscale=-1, yshift=0.5in]
      \filldraw[color=blue,opacity=0.2] (2, 3) -- (3, 2) -- (2, 2) --
      (1, 3) -- cycle;

      \filldraw[color=blue,opacity=0.5] (-1, 3) -- (2, 0) -- (2, 1)
      -- (0, 3) -- cycle;
      \filldraw[color=red,opacity=0.5] (0, 1) -- (1, 0) -- (0, 0) --
      cycle;

      \filldraw[color=orange,opacity=0.5] (-1, 3) -- (-1, 2) -- (0, 2)
      -- cycle;
      \filldraw[color=orange,opacity=0.5] (0, 3) -- (2, 1) -- (2, 2)
      -- (1, 3) -- cycle;

      \filldraw[color=green,opacity=0.5] (2, 3) -- (3, 2) -- (3, 3)
      -- cycle;
      \filldraw[color=green,opacity=0.5] (0, 2) -- (0, 1) -- (1, 0)
      -- (2, 0) -- cycle;

      \draw (0, 0) -- (0, 2) -- (-1, 2) -- (-1, 3);
      \draw[color=green!80!black] (-1, 3) -- node[below] {$\gamma_1$} (0, 3);
      \draw (0, 3) -- (2, 3);
      \draw[color=green!80!black] (2, 3) -- node[below] {$\gamma_3$} (3, 3);
      \draw (3, 3) -- (3, 2) -- (2, 2) -- (2, 0);
      \draw[color=red] (2, 0) -- node[above] {$\sigma_2$} (0, 0);
    \end{scope}
    \begin{scope}[thick,xshift=1in,yscale=-1, yshift=0.5in]
      \filldraw[color=red,opacity=0.2] (2, 3) -- (3, 2) -- (2, 2) --
      (1, 3) -- cycle;

      \filldraw[color=blue,opacity=0.5]  (0, 1) -- (1, 0) -- (0, 0) -- cycle;
      \filldraw[color=red,opacity=0.5] (-1, 3) -- (2, 0) -- (2, 1)
      -- (0, 3) -- cycle;
      \filldraw[color=green,opacity=0.5] (-1, 3) -- (-1, 2) -- (0, 2)
      -- cycle;
      \filldraw[color=green,opacity=0.5] (0, 3) -- (2, 1) -- (2, 2)
      -- (1, 3) -- cycle;
      
      \filldraw[color=orange,opacity=0.5] (2, 3) -- (3, 2) -- (3, 3)
      -- cycle;
      \filldraw[color=orange,opacity=0.5] (0, 2) -- (0, 1) -- (1, 0)
      -- (2, 0) -- cycle;

      \draw (0, 0) -- (0, 2) -- (-1, 2) -- (-1, 3);
      \draw[color=green!80!black] (-1, 3) -- node[below] {$\gamma_4$} (0, 3);
      \draw (0, 3) -- (2, 3);
      \draw[color=green!80!black] (2, 3) -- node[below] {$\gamma_6$} (3, 3);
      \draw (3, 3) -- (3, 2) -- (2, 2) -- (2, 0);
      \draw[color=red] (2, 0) -- node[above] {$\sigma_5$} (0, 0);
    \end{scope}
  \end{tikzpicture}
  \caption{Computation of $F_{\sfrac{\pi}{4}}$.}
  \label{fig:computingIET}
\end{figure}

From Figure~\ref{fig:computingIET} together with our coordinatization
described above, we see that $F_{\sfrac{\pi}{4}}$ is the following
piecewise map.
\[
F_{\sfrac{\pi}{4}}(x) = \begin{cases}
  x + \sfrac{5}{2} & \text{ if } x \in (0, \sfrac{1}{2}) \\
  x - \sfrac{1}{2} & \text{ if } x \in (\sfrac{1}{2}, 1) \\
  x + \sfrac{5}{2} & \text{ if } x \in (1, \sfrac{3}{2}) \\
  x - \sfrac{1}{2} & \text{ if } x \in (\sfrac{3}{2}, 2) \\
  x - \sfrac{3}{2} & \text{ if } x \in (2, \sfrac{5}{2}) \\
  x - \sfrac{1}{2} & \text{ if } x \in (\sfrac{5}{2}, 3) \\
  x - \sfrac{3}{2}  & \text{ if } x \in (3, \sfrac{7}{2}) \\
  x - \sfrac{1}{2} & \text{ if } x \in (\sfrac{7}{2}, 4)
\end{cases}
\]

Represented as a diagram showing how the intervals are permuted, the
IET is also presented in Figure~\ref{fig:IETdiagram}.

\begin{figure}
  \centering
  \begin{tikzpicture}
    \draw[very thick, color=red] (0, 0) -- (1.25, 0);
    \draw (0, -0.25) -- (0, 0.25);
    \draw (1.25, -0.25) -- (1.25, 0.25);
    \node at (0.625, 0.5) {A};
    \draw[very thick, color=blue] (1.25, 0) -- (2.5, 0);
    \draw (1.25, -0.25) -- (1.25, 0.25);
    \draw (2.5, -0.25) -- (2.5, 0.25);
    \node at (1.875, 0.5) {B};
    \draw[very thick, color=green] (2.5, 0) -- (3.75, 0);
    \draw (2.5, -0.25) -- (2.5, 0.25);
    \draw (3.75, -0.25) -- (3.75, 0.25);
    \node at (3.125, 0.5) {C};
    \draw[very thick, color=orange] (3.75, 0) -- (5.0, 0);
    \draw (3.75, -0.25) -- (3.75, 0.25);
    \draw (5.0, -0.25) -- (5.0, 0.25);
    \node at (4.375, 0.5) {D};
    \draw[very thick, color=purple] (5.0, 0) -- (6.25, 0);
    \draw (5.0, -0.25) -- (5.0, 0.25);
    \draw (6.25, -0.25) -- (6.25, 0.25);
    \node at (5.625, 0.5) {E};
    \draw[very thick, color=black] (6.25, 0) -- (7.5, 0);
    \draw (6.25, -0.25) -- (6.25, 0.25);
    \draw (7.5, -0.25) -- (7.5, 0.25);
    \node at (6.875, 0.5) {F};
    \draw[very thick, color=cyan] (7.5, 0) -- (8.75, 0);
    \draw (7.5, -0.25) -- (7.5, 0.25);
    \draw (8.75, -0.25) -- (8.75, 0.25);
    \node at (8.125, 0.5) {G};
    \draw[very thick, color=yellow] (8.75, 0) -- (10.0, 0);
    \draw (8.75, -0.25) -- (8.75, 0.25);
    \draw (10.0, -0.25) -- (10.0, 0.25);
    \node at (9.375, 0.5) {H};
    \draw[very thick, color=blue] (0, -3) -- (1.25, -3);
    \draw (0, -3.25) -- (0, -2.75);
    \draw (1.25, -3.25) -- (1.25, -2.75);
    \node at (0.625, -3.5) {B};
    \draw[very thick, color=purple] (1.25, -3) -- (2.5, -3);
    \draw (1.25, -3.25) -- (1.25, -2.75);
    \draw (2.5, -3.25) -- (2.5, -2.75);
    \node at (1.875, -3.5) {E};
    \draw[very thick, color=orange] (2.5, -3) -- (3.75, -3);
    \draw (2.5, -3.25) -- (2.5, -2.75);
    \draw (3.75, -3.25) -- (3.75, -2.75);
    \node at (3.125, -3.5) {D};
    \draw[very thick, color=cyan] (3.75, -3) -- (5.0, -3);
    \draw (3.75, -3.25) -- (3.75, -2.75);
    \draw (5.0, -3.25) -- (5.0, -2.75);
    \node at (4.375, -3.5) {G};
    \draw[very thick, color=black] (5.0, -3) -- (6.25, -3);
    \draw (5.0, -3.25) -- (5.0, -2.75);
    \draw (6.25, -3.25) -- (6.25, -2.75);
    \node at (5.625, -3.5) {F};
    \draw[very thick, color=red] (6.25, -3) -- (7.5, -3);
    \draw (6.25, -3.25) -- (6.25, -2.75);
    \draw (7.5, -3.25) -- (7.5, -2.75);
    \node at (6.875, -3.5) {A};
    \draw[very thick, color=yellow] (7.5, -3) -- (8.75, -3);
    \draw (7.5, -3.25) -- (7.5, -2.75);
    \draw (8.75, -3.25) -- (8.75, -2.75);
    \node at (8.125, -3.5) {H};
    \draw[very thick, color=green] (8.75, -3) -- (10.0, -3);
    \draw (8.75, -3.25) -- (8.75, -2.75);
    \draw (10.0, -3.25) -- (10.0, -2.75);
    \node at (9.375, -3.5) {C};
    \draw[->,thick,color=red] (0.625, -0.25) .. controls (0.625, -1) and (6.875, -2) .. (6.875, -2.75);
    \draw[->,thick,color=green] (3.125, -0.25) .. controls (3.125, -1) and (9.375, -2) .. (9.375, -2.75);
    \draw[->,thick,color=blue] (1.875, -0.25) .. controls (1.875, -1) and (0.625, -2) .. (0.625, -2.75);
    \draw[->,thick,color=purple] (5.625, -0.25) .. controls (5.625, -1) and (1.875, -2) .. (1.875, -2.75);
    \draw[->,thick,color=orange] (4.375, -0.25) .. controls (4.375, -1) and (3.125, -2) .. (3.125, -2.75);
    \draw[->,thick,color=cyan] (8.125, -0.25) .. controls (8.125, -1) and (4.375, -2) .. (4.375, -2.75);
    \draw[->,thick,color=black] (6.875, -0.25) .. controls (6.875, -1) and (5.625, -2) .. (5.625, -2.75);
    \draw[->,thick,color=yellow] (9.375, -0.25) .. controls (9.375, -1) and (8.125, -2) .. (8.125, -2.75);
  \end{tikzpicture}
  \caption{The interval exchange $F_{\frac{\pi}{4}}$. }
  \label{fig:IETdiagram}
\end{figure}

Notice that when $F_\theta(x)$ is applied to a point in one of the
$\gamma_i$ intervals at the base of some quad-T subsurface
$\mathscr{Q}^s$, the corresponding image in $\sigma_j$ of
$\mathscr{Q}^s$ is identified with some $\gamma_k$ of
$\mathscr{Q}^{s'}$ where $s'$ is $s$ with either one bit appended or
the last bit deleted.  In particular, the corresponding $\gamma_k$ in
$\mathscr{Q}^{s'}$ interval may have half or twice the length of the
$\gamma_k$ in $\mathscr{Q}^s$, depending on whether a bit was appended
or removed.  In order to iterate the IET, we must compose the map
$F_\theta$ with another map \label{ntn:PhiRenormalizationMap} $\Phi$ to normalize the coordinates.
Identifying $\Gamma$ and $\Sigma$ with one interval $I$ as before,
$\Phi$ is then a map from $I \times \mathscr{B}^* \to I \times
\mathscr{B}^*$.  We introduce some notation to make the map $\Phi$
easier to describe.  For a binary string $s \in \mathscr{B}^*$ we of
course let $s0$ and $s1$ mean the string with one bit appended to the
end of $s$; \label{ntn:nablas} $\nabla s$ means $s$ with its right-most bit
deleted; and \label{ntn:lambdas} $\lambda(s)$ is the right-most bit of $s$. 
This renormalization map $\Phi$ is independent of the
direction chosen and is easily seen to be the following.

\[
\Phi(x, s) = \begin{cases}
  (2x + \frac{1}{2} , s0) &\text{ if } x \in (0, \sfrac{1}{2})  \\
  (\frac{x}{2} + \frac{5}{4}, \nabla s) & \text{ if } x \in
  (\frac{1}{2}, \frac{3}{2}) \text{ and } \lambda(s) = 1 \\
  (\frac{x}{2} - \frac{1}{4}, \nabla s) & \text{ if } x \in
  (\frac{1}{2}, \frac{3}{2}) \text{ and } \lambda(s) = 0 \\
  (2x - \frac{5}{2}, s1) & \text{ if } x \in (\sfrac{3}{2}, 2) \\
  (2x - \frac{3}{2}, s1) & \text{ if } x \in (2, \sfrac{5}{2}) \\
  (\frac{x}{2} + \frac{3}{4}, \nabla s) & \text{ if } x \in
  (\sfrac{5}{2}, \sfrac{7}{2})  \text{ and } \lambda(s) = 1 \\
  (\frac{x}{2} + \frac{9}{4}, \nabla s) & \text{ if } x \in
  (\sfrac{5}{2}, \sfrac{7}{2})  \text{ and } \lambda(s) = 0 \\
  (2x - \frac{9}{2}, s0) & \text{ if } x \in (\sfrac{7}{2}, 4) \\
  (x, \epsilon) & \text{ if } x \in (\sfrac{1}{2}, \sfrac{3}{2}) \cup
  (\sfrac{5}{2}, \sfrac{7}{2}) \text{ and } s = \epsilon
\end{cases}
\]

Given a starting point $x_0 \in \Gamma^{s_0}$, the geodesic ray in
direction $\theta$ crosses the quad-T subsurfaces in the sequence of
points given by
\[
  (x_{n+1}, s_{n+1}) = \Phi(F_\theta(x_n), s_n)
\]
where $x_n \in \Gamma^{s_n}$.

\section{Existence of Linear Approaches}
\label{sec:linearApproachesExist}
We now build on the content introduced in
Section~\ref{sec:IETs} to show that each elusive singularity has a linear approach.  Given an elusive
singularity $x$ with address $\alpha(x) = e_1 e_2 e_3 \ldots$, we consider several cases:  
\begin{enumerate}
\item there exists $N$ such that $e_{n} = 0$ for all $n\geq N$ or $e_{n} = 1$  for all $n\geq N$ (i.e., $\alpha(x)$  eventually consists of all zeros or all ones),
\item  there exists $N$ such that $e_{2n} = 0$ and $e_{2n+1} = 1$ for all $n\geq N$ (similarly, $e_{2n} = 1$ and $e_{2n+1} = 0$ for all $n\geq N$).  In other words, $\alpha(x)$ ends in a repeating pattern $01010\ldots $ (or $10101\ldots$), and
\item all other cases.
\end{enumerate}

\subsection{The address ends in all zeros or all ones}
We first consider the case of the address consisting entirely of
zeros.  Once this case is understood, the case of any other address ending in all
zeros or all ones will be easy to prove.

Consider the direction $\theta = \tan^{-1}\left(\frac{12}{19}\right)$.
A simple, though tedious, calculation shows the geodesic emitted from
the point $\frac{11}{4} \in \Gamma^\epsilon$ in the direction $\theta$ next intersects the point
$\frac{11}{4} \in \Gamma^0$.  Or, in terms of the interval exchanges,
$F_{\tan^{-1}\left(\sfrac{12}{19}\right)}\left(\frac{11}{4}\right) =
\frac{17}{8}$ and
$\Phi\left(\frac{17}{8}, \epsilon\right) = \left(\frac{11}{4},
  0\right)$.  Alternatively, the billiard in the original $T$-shaped
polygon which starts at the point $\left(\frac{1}{4}, 0\right)$, with
the bottom-most left-corner placed at $(0, 0)$, first reaches the top
of the $T$ at the point $\left(-\frac{3}{8}, \frac{3}{2}\right)$.  See
Figure~\ref{fig:all_zero_trajectory} for the billiard interpretation.

\begin{figure}[h!]
  \centering
  \begin{tikzpicture}[scale=3]
    \draw (0, 0) -- (1, 0) -- (1, 1) -- (1.5, 1) -- (1.5, 1.5);
    \draw[dashed] (1.5, 1.5) -- (1, 1.5);
    \draw (1, 1.5) -- (0, 1.5);
    \draw[dashed] (0, 1.5) -- (-0.5, 1.5);
    \draw (-0.5, 1.5) -- (-0.5, 1) -- (0, 1) -- (0, 0);

    \draw[red] (0.25, 0) -- (1, 18/38) -- (-0.5, 54/38) -- (-0.375, 1.5);
  \end{tikzpicture}
  \caption{The beginning of a linear approach to $0000\ldots$ in the billiard.}
  \label{fig:all_zero_trajectory}
\end{figure}

Applying the interval exchange again (equivalently, allowing the
billiard to bounce up to the next level where a scaled $T$ is attached) gives
the point $\frac{11}{4} \in \Gamma^{00}$.  In general, this ray always
passes from the quad-T with address $\underbrace{000\ldots0}_{n \text{ zeros}}$ to the same point in the quad-T whose address has one
more zero.  This corresponds to a linear approach to the elusive
singularity with address $0000\ldots$ in the T-fractal surface.  A simple
horizontal reflection thus produces an approach to the elusive
singularity with address $1111\ldots$.  Repeating this argument, but
using the ray emanated from the corresponding point in the quad-T with
address $s$ gives a linear approach to the elusive singularity with
address $s0000\ldots$ or $s1111\ldots$.

\subsection{The address ends in $01$ repeating}
\label{subsec:FIET_irrational_direction}
It is shown in \cite{LapMilNie} that a billiard path in the
$T$-fractal billiard beginning on the base\footnote{Recall that the
  base of $\tfractal$ is the interval $[0,1]$, this being different
  from the coordinates on $\tsurf$.} at
\[x_0 = 1700\sqrt{2}/2 -1202 \]
\noindent with an initial direction
$\theta = \pi - \arctan(\sqrt{2}/34)$ will reach an elusive point of
the $T$-fractal billiard without intersecting any singularities of the
billiard table (i.e., corners). In $\tsurf$, such an elusive point
corresponds to an elusive singularity with an address of
$01010\ldots$.  For the convenience of the reader, we provide an
alternative, simpler description of a geodesic which reaches the
elusive singularity with address $010101...$.  Once this has been
established we can, as in the case of an
address ending in all zeros, we can obtain a geodesic ray reaching an
elusive singularity with any address ending in alternative $010101...$
by translating the initial point of the ray to the quad-T addressed by
the desired prefix.

Simply consider the billiard emanating from the midpoint of the base
of the $T$-fractal with initial slope $6/7$.  A simple calculation
shows that this billiard will first hit the right-hand side of the $T$
at $3/7$ above the base, then reflect and travel in a straight line
until reaching the midpoint of the left-hand arm of the $T$.  See
Figure~\ref{fig:alternating}.

\begin{figure}[h!]
    \begin{tikzpicture}[scale=3]
    \draw (0, 0) -- (1, 0) -- (1, 1) -- (1.5, 1) -- (1.5, 1.5);
    \draw[dashed] (1.5, 1.5) -- (1, 1.5);
    \draw (1, 1.5) -- (0, 1.5);
    \draw[dashed] (0, 1.5) -- (-0.5, 1.5);
    \draw (-0.5, 1.5) -- (-0.5, 1) -- (0, 1) -- (0, 0);

    \draw[red] (0.5, 0) -- (1, 3/7) -- (-0.25, 3/2);
  \end{tikzpicture}
  \caption{A billiard which reaches the elusive singularity with
    address $010101...$.}
  \label{fig:alternating}
\end{figure}

As the next $T$ is a scaled copy of the base, but the slope is the
negative of our initial slope, the billiard will enter the right-hand
arm of the next $T$ at the midpoint, but now with the original slope.
Continuing in this way, the billiard moves all the way up the
$T$-fractal moving left, then right, then left, then right, and so on,
as it makes its way up the $T$-fractal.

\subsection{All other cases}
Now suppose that $\alpha(x) = e_1e_2e_3 \ldots$ is the address of any elusive
singularity not already discussed in the previous two cases.  We will
use the IET corresponding to direction $\theta = \sfrac{\pi}{4}$
calculated in Section~\ref{sec:IETs} to determine a starting point in
$\mathcal{Q}^{\epsilon}$ for a ray which in direction $\theta$ gives a
linear approach to the elusive singularity.

We first make an observation about the sequence $(x_{n+1}, s_{n+1}) =
\Phi(F(x_n), s_n)$.  Identifying the $\Gamma$ interval with $[0, 4]$
as above, as was the case, we have that $\gamma_2 = [\sfrac{1}{2}, \sfrac{3}{2}]$ and $\gamma_5 =
[\sfrac{5}{2}, \sfrac{7}{2}]$.  Then, from the definition of
$F_{\frac{\pi}{4}}$ and $\Phi$, we see that 
\begin{equation}
  \label{eqn:nexts}
s_{n+1} = \begin{cases}
  s_n0 & \text{ if } x_n \in \left(\frac{1}{2}, \frac{3}{2}\right)
  \setminus \{1\} \\
  s_n1 & \text{ if } x_n \in \left(\frac{5}{2}, \frac{7}{2}\right)
  \setminus \{3\}
\end{cases}
\end{equation}
Furthermore,
\begin{equation}
  \label{eqn:nextx}
x_{n+1} \in \begin{cases}
  \left(\frac{1}{2}, \frac{3}{2}\right) & \text{ if } x_n \in
  \left(\frac{1}{2}, 1\right) \\
  \left(\frac{5}{2}, \frac{7}{2}\right) & \text{ if } x_n \in \left(1,
  \frac{3}{2}\right) \\
  \left(\frac{1}{2}, \frac{3}{2}\right) & \text{ if } x_n \in
  \left(\frac{5}{2}, 3\right) \\
  \left(\frac{5}{2}, \frac{7}{2}\right) & \text{ if } x_n \in
  \left(3, \frac{7}{2}\right) \\
\end{cases}
\end{equation}
This means any starting point $x_0$ chosen to be in
$\left(\frac{1}{2}, \frac{3}{2}\right) \cup \left(\frac{5}{2},
  \frac{7}{2}\right)$ will give a sequence of $s_n$'s which always has
a single bit appended at each iteration and so there is a well-defined
limiting infinite binary string, $s_\infty = (s_n)_{n\in\mathbb{N}}$.
Thus our goal is to show that given an infinite binary string
$\alpha(x) = e_1 e_2 e_3 \ldots$, the string $s_\infty$ generated from
$(x_0, \epsilon)$ is equal to the given $e$.

To accomplish this we will construct a function
$\kappa : \mathscr{B}^\mathbb{N} \to [0, 4]$ such that the sequence of ordered pairs
$(x_{n+1}, s_{n+1}) = \Phi(F(x_n), s_n)$ generated from
$(x_0, \epsilon) = (\kappa(\alpha(x)), \epsilon)$ satisfies
$x_n = \kappa(\delta^n(\alpha(x)))$ and $s_n = s_{n-1}e_n$, where
$\delta : \mathscr{B}^{\mathbb{N}} \to \mathscr{B}^{\mathbb{N}}$ is
the shift map, $\delta(b_1 b_2 b_3 \ldots) = b_2 b_3 b_4 \ldots$.  To
define $\kappa$, we suppose the address $\alpha(x)$ is written in terms of
blocks of zeros and ones as
\[
\alpha(x) = 0^{n_1}1^{n_2}0^{n_3} \ldots \text{ or } \alpha(x) = 1^{n_1}0^{n_2}1^{n_3}\ldots,
\]
\noindent where $n_i>0$ for $i\geq 1$.   We then define $\eta(\alpha(x))$ as
\[
\eta(\alpha(x)) = \sum_{i=1}^\infty \left(\frac{1}{2}\right)^{\sum_{j=1}^i n_j}.
\]
Notice that $\eta(\alpha(x)) \in \left(\frac{1}{2}, 1\right)$ if $n_1 = 1$ and that
$\eta(\alpha(x)) \in \left(0, \frac{1}{2}\right)$ if $n_1 > 1$.  Now we define
\[
\kappa(\alpha(x)) = \begin{cases}
  \frac{1}{2} + \eta(\alpha(x)) & \text{ if } e_1 = 0 \\
  \frac{5}{2} + \eta(\alpha(x)) & \text{ if } e_1 = 1.
\end{cases}
\]

\begin{proposition}
  \label{prop:linear_approaches}
  If $\alpha(x) = e_1 e_2 e_3 \ldots$ is any address of an elusive singularity
  not ending in all $0$, all $1$, or in $01$-repeating, then the  
  geodesic ray emitted from $\kappa(\alpha(x))$ in the direction $\theta =
  \frac{\pi}{4}$ will be a linear approach to the elusive singularity with address $e$.
\end{proposition}
\begin{proof}
  Let $x_0 = \kappa(\alpha(x))$ and $s_0 = \epsilon$.  Let $(x_n, s_n)$ be the
  sequence of pairs of points and finite binary strings defined by
  $(x_{n+1}, s_{n+1}) = \Phi(F(x_n), s_n)$.  We first make a few
  simple observations about $x_0$, $x_1$, $s_1$, and $s_2$ based on
  the first two bits of $e$.  In particular, we consider four cases
  corresponding to the four possibilities of the first two bits of
  $e$.  The computations we perform in each case are all very
  similar to one another, so we only explicitly calculate the $e_1e_2
  = 01$ and $e_1e_2 = 00$ cases, and leave it to the reader to make
  the corresponding changes for the two remaining cases.  We then prove via induction that the elusive singularity with address $\alpha(x)$ is reached by the geodesic ray emitted from $\kappa(\alpha(x))$ in the direction of $\theta=\pi/4$.  

  Suppose $e_1e_2 = 01$ and notice this means $n_1 = 1$.  We then compute $x_0 = \kappa(
 \alpha(x))$, $F_{\pi/4}(x_0)$ and $\Phi(F_{\pi/4}(x_0), \epsilon)$ as follows:
  \begin{align*}
  	x_0 &= \kappa(\alpha(x)) \\
  			&=\frac{1}{2} + \eta(\alpha(x)) \\
	 		&= \frac{1}{2} + \sum_{i=1}^\infty  \left(\frac{1}{2}\right)^{\sum_{j=1}^i n_j}\\
	 		&= \frac{1}{2}+\frac{1}{2} +\sum_{i=2}^\infty \left(\frac{1}{2}\right)^{\sum_{j=1}^i n_j},
	\end{align*}
	\noindent giving us that $x_0\in (1,\sfrac{3}{2})$.  Then,
	\begin{align*}
		F_{\pi/4}(x_0) &= x_0+\frac{5}{2}\\
								&= \frac{7}{2} + \sum_{i=2}^\infty \left(\frac{1}{2}\right)^{\sum_{j=1}^i n_j},
	\end{align*}
	\noindent giving us that $F_{\pi/4} \in (\sfrac{7}{2},4)$.  Consequently, $\Phi(F_{\pi/4}(x_0),\epsilon) = (2F_{\pi/4}(x_0)-\sfrac{9}{2}, 0)$.  That is,
	\begin{align*}
		x_1 &= 2 F_{\pi/4}(x_0) - \frac{9}{2}\\
				&=7+2\sum_{i=2}^\infty \left(\frac{1}{2}\right)^{\sum_{j=1}^i n_j} -\frac{9}{2}\\
   		&= \frac{5}{2} + \sum_{i=2}^\infty \left(\frac{1}{2}\right)^{-1 + \sum_{j=1}^i n_j}\\
    	&= \frac{5}{2} + \sum_{i=2}^\infty 	\left(\frac{1}{2}\right)^{\sum_{j=2}^i n_j},
  	\end{align*}
 \noindent giving us that $x_1\in (\sfrac{5}{2},\sfrac{7}{2})$.  Now, further computation shows that 
 	\[ x_2 \in \begin{cases}
 		(\sfrac{3}{2},2) & \text{if } x_1 \in (3,\sfrac{7}{2})\\
 		(2,\sfrac{5}{2}) & \text{if } x_1 \in (\sfrac{5}{2}, 3). 		
 	\end{cases}
 	\]
 In either case, $s_2 = 01 = e_1e_2$.  Additionally, we see from the definition of $\kappa(\alpha(x))$ that $x_1 = \kappa(\delta(\alpha(x)))$.  

Now suppose that $e_1e_2 = 00$.  Then, $n_1 > 1$, meaning that $\eta(\alpha(x))\in (0,\sfrac{1}{2})$.  Under such an assumption, we compute $x_0 = \kappa(\alpha(x))$, $F_{\pi/4}(x_0)$ and $\Phi(F_{\pi/4}(x_0), \epsilon)$ as follows.  Since $n_1>1$, $\eta(\alpha(x))$ does not simplify as before:
  \begin{align*}
    x_0 &= \kappa(\alpha(x))\\
    		&=\frac{1}{2} +\eta(\alpha(x)).
   \end{align*}
   \noindent Since $\eta(\alpha(x)) \in (0,\sfrac{1}{2})$, we have that $x_0\in (\sfrac{1}{2},1)$, meaning that $F_{\pi/4}(x_0) = \eta(\alpha(x))$ and $\Phi(F_{\pi/4}(x_0),\epsilon) = (2F_{\pi/4}(x_0) +\sfrac{1}{2},0)$. That is,
   \begin{align*}
    x_1 &= 2 F_{\pi/4}(x_0)+\frac{1}{2}\\
    		&= 2\eta(\alpha(x))+\frac{1}{2}\\
    		&= \sum_{i=1}^\infty \left(\frac{1}{2}\right)^{-1+n_1+\sum_{j=2}^i n_j}+\frac{1}{2},
   \end{align*} 
 \noindent giving us that $x_1\in (\sfrac{1}{2},\sfrac{3}{2})$.  From the definition of $\kappa$, we see again that $x_1= \kappa(\delta(\alpha(x)))$.  Further computation shows that $x_2\in (\sfrac{1}{2},\sfrac{3}{2})$ and $s_2 = 00$, meaning that $s_2 = e_1e_2$, as claimed.  
  
  A similar set of computations---left to the reader---also shows that if $e_1e_2$ equals
  $10$ or $11$, then $s_2 = e_1e_2$ and $x_1 = \kappa(\delta(\alpha(x)))$.

Suppose now that there exists $N\in \mathbb{N}$ such that for every natural number $m\leq N$,   $x_m = \kappa(\delta^m(\alpha(x)))$ and $s_m = e_1\ldots e_m$.  Then, 
\begin{align*}
(x_{N+1},s_{N+1}) 	&= \Phi(F_{\pi/4}(x_N), s_N)\\
								&= \Phi(F_{\pi/4}(\kappa(\delta^N(\alpha(x))), s_N)\\
								&= \begin{cases}
									\Phi(F_{\pi/4}(\sfrac{1}{2} +\eta(\delta^N(\alpha(x)))), s_N) &\text{if } e_{N+1} = 0\\
									\Phi(F_{\pi/4}(\sfrac{5}{2} +\eta(\delta^N(\alpha(x)))), s_N) &\text{if } e_{N+1} = 1
								\end{cases}
\end{align*}

Suppose $e_{N+1} = 0$.  Then, 
\[F_{\pi/4}(\sfrac{1}{2} +\eta(\delta^N(\alpha(x))))
	=	\begin{cases}
		 \sfrac{1}{2}+\eta(\delta^N(\alpha(x)))-\sfrac{1}{2} &\text{if } \eta(\delta^N(\alpha(x)))\in (1,\sfrac{1}{2}) \\
		 \sfrac{1}{2}+\eta(\delta^N(\alpha(x)))+\sfrac{5}{2} &\text{if } \eta(\delta^N(\alpha(x)))\in (\sfrac{1}{2},1) 
		\end{cases}
\]	

Suppose that $\eta(\delta^N(\alpha(x)))\in (0,\sfrac{1}{2})$.  Then,
\[
	\Phi(\eta(\delta^N(\alpha(x))),s_N) = (2\eta(\delta^N(\alpha(x)))+\sfrac{1}{2}, s_N0)
\]
\noindent and
\begin{align*}
x_{N+1} &= 2\eta(\delta^N(\alpha(x)))+\sfrac{1}{2}\\
		&= \eta(\delta(\delta^N(\alpha(x))))+\sfrac{1}{2}\\
		&= \eta(\delta^{N+1})+\sfrac{1}{2}\\
		&= \kappa(\delta^{N+1}(\alpha(x))),
\end{align*}

\noindent as claimed.  Now suppose that $\eta(\delta^N(\alpha(x)))\in (\sfrac{1}{2},1)$.  Then, 

\[
F_{\pi/4}(\sfrac{1}{2}+\eta(\delta^N(\alpha(x))))=3+\eta(\delta^N(\alpha(x)))
\]
\noindent  and

\begin{align*}
\Phi(3+\eta(\delta^N(\alpha(x))), s_N) &= (\sfrac{3}{2} + 2\eta(\delta^N(\alpha(x))), s_N0),
\end{align*}

\noindent giving us that 

\begin{align*}
x_{N+1} 	&= \sfrac{3}{2}+\eta(\delta(\delta^N(\alpha(x))))+1\\
				&= \sfrac{5}{2}+\eta(\delta^{N+1}(\alpha(x)))\\
				&=\kappa(\delta^{N+1}(\alpha(x))).			
\end{align*}

\noindent In both cases, $s_{N+1} = s_{N}0 = e_1e_2\ldots e_Ne_{N+1}$.

Similar calculations show that when $e_{N+1} = 1$, $x_{N+1} = \kappa(\delta^{N+1}(\alpha(x)))$ and $S_{N+1}=e_1\ldots e_{N+1}$.  Therefore, by induction, $x_n = \kappa(\delta^n(\alpha(x)))$ and $s_n = e_1\ldots e_n$ for all $n$.  

Therefore, $\lim_{n\to\infty} s_n = e$ and the geodesic beginning at $x_0=\kappa(\alpha(x))$ in the direction of $\theta = \sfrac{\pi}{4}$ will reach the elusive singularity with address $e$ in the limit.

%
\end{proof}

We remark that the proof of Proposition~\ref{prop:linear_approaches}
does not extend to the two cases discussed previously, when $e$ ends
in repeating $0$, $1$, or $01$.  In those cases, the series will converge to a dyadic rational (adopting the convention
that in the all zero or all ones cases  $\sum n_i = \infty$ and the
corresponding term $\left(\frac{1}{2}\right)^{\sum n_j}$ equals
zero).  In direction $\theta = \frac{1}{4}$, however, any ray emitted
from a point $x_0$ which is a dyadic rational is easily seen to reach
a finite angle conical singularity in the T-fractal surface.

Combining the three cases above together we have thus proven the
following.

\begin{theorem}
\label{thm:every_elusive_linear_approach}
  For each elusive singularity of the T-fractal surface there exists a
  linear approach to that singularity.
\end{theorem}

\section{The metric geometry of $\tsurf$}
\label{sec:theMetricCompletionOfTsurf}



We now record some observations about the metric geometry of $\tsurf$
which will be used in
Section~\ref{sec:elusiveSingularitiesWildSingularities} to answer some
basic questions about the elusive singularities.
  
To simplify the language in some of the arguments to come, we
introduce an operation on pairs of binary strings.  Given
$s, t \in \mathscr{B}^*$, let \label{ntn:wedge} $s \wedge t$ denote
the longest substring of both $s$ and $t$ such that $s$ (resp., $t$)
is obtained by appending a string to the right-hand end of $s$ (resp.,
$t$).  That is, there exist strings $s'$ and $t'$ such that
$s = (s \wedge t) s'$ and $t = (s \wedge t) t'$.  For example, if
$s = 1001101$ and $t = 1001001$, then $s \wedge t = 1001$.  Notice
that $s \wedge t$ will be the empty string if the first (left-most)
character of $s$ and $t$ disagree; e.g., $1 \wedge 0 = \epsilon$.

Let $\Q^s$ and $\Q^t$ be two quad-T surfaces in $\tsurf$.  If
$|\mu|$ represents the length of a geodesic with one endpoint in
$\Q^s$ and one endpoint in $\Q^t$, then we define the distance
between $\Q^s$ and $\Q^t$ to be
\[
  \dist(\Q^s,\Q^t)\coloneqq \inf\{|\mu|:\mu \text{ is a broken geodesic from
  } \Q^s \text{ to } \Q^t\}.
\]

We want to emphasize that the distance between $\Q^s$ and $\Q^t$
\emph{is not} a metric on the set of quad-T subsurfaces of
$\tsurf$.  

\begin{lemma}
  \label{lem:connecting_quadts}
  Given any distinct binary strings $s$ and $t$, there exists a broken
  geodesic connecting $\Q^s$ and $\Q^t$ which passes through each
  intermediate quad-T subsurface exactly once.  That is, if this
  broken geodesic is parametrized as $\gamma : [0, \ell] \to \tsurfr$,
  then for each binary string $u$ which has $s \wedge t$ as a prefix,
  $\gamma^{-1}(\Q^u)$ is connected.  Furthermore, the length of the
  portion of $\gamma$ in $\Q^u$ is less than the diameter of $\Q^u$.
\end{lemma}
\begin{proof}
  We note that a broken geodesic connecting distinct boundary
  components of $\Q$ is easily constructed, as in
  Figure~\ref{fig:broken_boundary}.  That there exists a broken
  geodesic with length less than the diameter of $\Q$ follows from the
  definition of the diameter.
  
  \begin{figure}[h!]
    \centering
    \begin{tikzpicture}[scale=0.66]
      \begin{scope}[thick,xshift=1.25in]
        \draw (0, 0) -- node[scale=0.66,left] {$E$} (0, 2) 
        -- node[scale=0.66,below] {$F$} (-1, 2)
        -- node[scale=0.66,left] {$G$} (-1, 3);
        \draw (0, 3) -- node[scale=0.66,above] {$A$} (2, 3);
        \draw (3, 3) -- node[scale=0.66,right] {$B$} (3, 2)
        -- node[scale=0.66,below] {$C$} (2, 2)
        -- node[scale=0.66,right] {$D$} (2, 0);

        \draw[\sigmaColor] (-1, 3) -- (0, 3);
        \draw[\sigmaColor] (2, 3) -- (3, 3);
        \draw[\gammaColor] (2, 0) -- (0, 0);

      \end{scope}

      \begin{scope}[thick,xshift=-1.25in]
        \draw (0, 0) -- node[scale=0.66,left] {$D$} (0, 2)
        -- node[scale=0.66,below] {$J$} (-1, 2)
        -- node[scale=0.66,left] {$B$} (-1, 3);
        \draw (0, 3) -- node[scale=0.66,above] {$H$} (2, 3);
        \draw (3, 3) -- node[scale=0.66,right] {$G$} (3, 2)
        -- node[scale=0.66,below] {$I$} (2, 2)
        -- node[scale=0.66,right] {$E$} (2, 0);

        \draw[\sigmaColor] (-1, 3) -- (0, 3);
        \draw[\sigmaColor] (2, 3) -- (3, 3);
        \draw[\gammaColor] (2, 0) -- (0, 0);
        
        \filldraw[red] (-0.5, 3) circle (0.025in);
        \draw[red] (-0.5, 3) -- (-0.5, 2);
      \end{scope}

      \begin{scope}[thick,xshift=-1.25in,yscale=-1, yshift=0.5in]
        \draw (0, 0) -- node[scale=0.66,left] {$N$} (0, 2)
        -- node[scale=0.66,above] {$J$} (-1, 2)
        -- node[scale=0.66,left] {$M$} (-1, 3);
        \draw (0, 3) -- node[scale=0.66,below] {$H$} (2, 3);
        \draw (3, 3) -- node[scale=0.66,right] {$L$} (3, 2)
        -- node[scale=0.66,above] {$I$} (2, 2)
        -- node[scale=0.66,right] {$K$} (2, 0);

        \draw[\gammaColor] (-1, 3) -- (0, 3);
        \draw[\gammaColor] (2, 3) -- (3, 3);
        \draw[\sigmaColor] (2, 0) -- (0, 0);

        \draw[red] (-0.5, 2) -- (-0.5, 2.5) -- (3, 2.5);
      \end{scope}
      
      \begin{scope}[thick,xshift=1.25in,yscale=-1, yshift=0.5in]
        \draw (0, 0) -- node[scale=0.66,left] {$K$} (0, 2)
        -- node[scale=0.66,above] {$F$} (-1, 2)
        -- node[scale=0.66,left] {$L$} (-1, 3);
        \draw (0, 3) -- node[scale=0.66,below] {$A$} (2, 3);
        \draw (3, 3) -- node[scale=0.66,right] {$M$} (3, 2)
        -- node[scale=0.66,above] {$C$} (2, 2)
        -- node[scale=0.66,right] {$N$} (2, 0);

        \draw[\gammaColor] (-1, 3) --  (0, 3);
        \draw[\gammaColor] (2, 3) --  (3, 3);
        \draw[\sigmaColor] (2, 0) --  (0, 0);

        \draw[red] (-1, 2.5) -- (-0.5, 2.5) -- (-0.5, 3);
        \filldraw[red] (-0.5, 3) circle (0.025in);
      \end{scope}
    \end{tikzpicture}
    \caption{A broken geodesic connecting two boundary components of $\Q$.}
    \label{fig:broken_boundary}
  \end{figure}

  We now simply concatenate broken geodesics joining midpoints of
  boundary commponents of the quad-T subsurfaces between $\Q^s$ to
  $\Q^{s \wedge t}$, and then join the midpoints of boundary
  components connecting $\Q^{s \wedge t}$ to $\Q^t$, joining the
  midpoints of the boundary components in these two paths leading from
  $\Q^{s \wedge t}$ to one another by another broken geodesic.
\end{proof}

\begin{lemma}
  \label{lem:lower_bound_on_dist_Qs_Qsb2}
  Let $s \in \mathscr{B}^*$ with $|s|\geq 0$. Then
  $\dist(\Q^s, \Q^{sb_1b_2})$ is bounded below by
  $\frac{1}{2^{1+|s|}}$.
\end{lemma}

\begin{proof}
  Consider $s\in \mathscr{B}^*$.  Then $\dist(\Q^s,\Q^{sb_1}) = 0$,
  since they share a boundary component.  Since $\Q^s$ is $\Q$ scaled
  by $\frac{1}{2^{|s|}}$, it follows that any geodesic passing through
  $\Q^s$ must have a length greater than $\frac{1}{2^{1+|s|}}$.
\end{proof}

\begin{lemma}
  \label{lemma:qt_lowerbound}
  Let $s, t \in \mathscr{B}^*$ be distinct binary strings.  Then
  $\dist(\Q^s,\Q^t)$ will be zero if and only if $t$ equals $s$ with
  one more bit appended to the right \emph{(}or $s$ equals $t$ with
  one more bit appended to the right\emph{)}.  In all other cases the
  distance is bounded below by a positive constant depending on $|s|$
  and $|t|$.
\end{lemma}

\begin{proof}
  Suppose, without loss of generality, that $|s|\leq |t|$.  Further
  suppose that $s\wedge t = s$, where $t=st'$, such that $|t'| = 1$,
  $t'\in\mathscr{B}^*$.  Then $\Q^s$ and $\Q^t$ connect at a boundary
  component, and so $\dist(\Q^s,\Q^t) = 0$.
  
  Conversely, suppose that $s\wedge t = g$ such that either 
  
  \begin{enumerate}
  \item $g=s$, $t=st'$, $t'\in\mathscr{B}^*$ such that $|t'|\geq 2$ or
  \item $|g|<|s|\leq |t|$.
  \end{enumerate}

  Consider Case $(1)$: $g=s$, $t=st'$ such that $|t'|\geq 2$.   Let $t'=b_1\ldots b_n$. By Lemma  \ref{lem:lower_bound_on_dist_Qs_Qsb2}, 
  \begin{align}
    \notag \dist(\Q^{sb_0\ldots b_i},\Q^{sb_0\ldots b_{i+2}}) &\geq \frac{1}{2^{i+1+|s|}}
  \end{align}
  \noindent for all $0\leq i \leq n-2$ (where $b_0 = \epsilon$). Therefore,
  \begin{align}
    \notag \dist(\Q^s,\Q^t) & \geq \sum_{i=0}^{n-2} \dist(\Q^{sb_0\ldots b_i}, \Q^{sb_0\ldots b_{i+2}}),	
  \end{align}

  \noindent giving us  that 

  \begin{align}
    \notag \dist(\Q^s,\Q^t) & \geq \sum_{i=0}^{n-2} \frac{1}{2^{i+1+|s|}}\\
    \notag	&=\frac{1}{2^{|s|}}\sum_{i=1}^{n-1} \frac{1}{2^i}\\
    \notag	&=\frac{1}{2^{|s|}}\frac{2^{n-1}-1}{2^{n-1}}.
  \end{align}

  Since $|t|-|s| = n$, we have that 

  \begin{align}
    \notag \dist(\Q^s,\Q^t) & \geq \frac{1}{2^{|s|}}\frac{2^{|t|-|s|-1}-1}{2^{|t|-|s|-1}}\\
    \notag	&= \frac{2^{|t|-|s| - 1}-1}{2^{|t|-1}}.
  \end{align}


  
  Now, if $|s| = |t| = 1$ with $s \neq t$ (i.e., $s=0$ and $t=1$ or
  vice versa), then $\dist(\Q^s,\Q^t) \geq 1$.  This follows from
  Lemma \ref{lem:lower_bound_on_dist_Qs_Qsb2}, since any geodesic
  traversing both $\Q^s$ and $\Q^t$ necessarily passes through
  $\Q^\epsilon$.  In general, for $s\neq t$, $|s|=|t|$ and
  $|g|=|s|-1=|t|-1$, we have that
  $\dist(\Q^s,Q^t)\geq 2^{-|s \wedge t|}$, since any geodesic
  traversing $\Q^s$ and $\Q^t$ must necessarily traverse $\Q^g$.
   
  Proceeding now under the assumption that $|g|<|s|\leq |t|$, by Case 1, 
  \begin{align}
    \notag \dist(\Q^g,\Q^t) & = \frac{2^{|t|-|g| - 1}-1}{2^{|t|-1}},
  \end{align}

  \noindent and

  \begin{align}
    \notag \dist(\Q^g,\Q^s) &= \frac{2^{|s|-|g| - 1}-1}{2^{|s|-1}}.
  \end{align}

  Since the geodesic $\mu$ connecting $\Q^s$ and $Q^t$ must necessarily pass through $\Q^g$, we have that 

  \begin{align}
    \notag \dist(\Q^s,\Q^t) &\geq 	\frac{2^{|s|-|g| - 1}-1}{2^{|s|-1}} + \frac{2^{|t|-|g| - 1}-1}{2^{|t|-1}} + \frac{1}{2^{|g|}}.
  \end{align}
  %
  %
  %
\end{proof}

Similarly, there is an upper bound on how close two points contained
in a quad-T subsurface may be from one another. 

\begin{lemma}
  \label{lemma:qt_upperbound}
  For each $s \in \mathscr{B}^*$, the diameter of the quad-T
  subsurface $\Q^s \subseteq \tsurf$ is bounded above by
  $5\cdot 2^{|s|}$.
\end{lemma}
\begin{proof}
  We note that $T$-shaped polygons defining the quad-T surface $\Q$
  can be fit inside a $4 \times 3$ Euclidean rectangle.  The distance
  between any two points on $\Q$ is then less than the maximum
  distance between two points inside this rectangle which is simply
  $5$.  (The distance between two points in $\Q$ will often be
  strictly less than the distance between the corresponding points in
  the Euclidean rectangle because of the edge identifications.)  We
  now simply note that each $\Q^s$ subsurface is a rescaling of $\Q$
  by $2^{-|s|}$.
\end{proof}


\begin{lemma}
  \label{lemma:branch_diameter}
  The diameter of each branch $\mathcal{B}^s$ of $\tsurf$ is bounded
  above by $15 \cdot 2^{-|s|}$.
\end{lemma}
\begin{proof}
  Let $x, y \in \mathcal{B}^s$ and suppose $x \in \Q^{st_1}$ and
  $y \in \Q^{st_2}$.  As we are trying to obtain an upper bound on
  $d(x, y)$, we may replace $x$ and $y$ with other points in
  $\Q^{st_1}$ and $\Q^{st_2}$ that are further away than the
  originally chosen $x$ and $y$ if necessary.  In particular, since
  $\dist(\Q^{st_1}, \Q^{st_2}) \leq \dist(\Q^{st_1}, \Q^{st_2\tau})$
  for any $\tau \in \mathscr{B}^*$, we may suppose that
  $|t_1| = |t_2|$.

  We get an upper bound on $d(x, y)$ by finding a geodesic from $x$
  down to a point in $\Q^s$, and then from this point back up to $y$.
  The length of this geodesic in each of the intermediate quad-T
  subsurfaces has length no greater than the diameter of the quad-T.
  Hence, we sum these diameters to obtain the following, supposing
  $|t_1| = |t_2| = n$:
  \begin{align*}
    d(x, y) \leq& \frac{5}{2^{|s|}} + 2 \sum_{k=1}^n \frac{5}{2^{|s| +
                  k}} \\
    =& \frac{5}{2^{|s|}} \left(1 + \sum_{k=0}^{n-1}
       \frac{1}{2^{k}}\right) \\
    =& \frac{5}{2^{|s|}} \left(1 + 2\left(1 - 2^{-n}\right)\right).
  \end{align*}
  This diameter increases as $n$ increases (i.e., as the points move
  further up the branch $\mathcal{B}^s$), and taking the limit as
  $n \to \infty$ gives the inequality.
\end{proof}

To describe points of $\esurf$ we need to consider Cauchy sequences of
points of $\tsurfr$ which do not converge in $\mathcal{B}^\epsilon$.
We will show that equivalence classes of these Cauchy sequences, and
hence, points of $\esurf$, can be thought of as infinite binary strings
where the bits of the string tell us how to climb from the branch
$\mathcal{B}^\epsilon$ up to an elusive point.  The proof of this fact
is broken down into several steps presented as lemmas below.

To ease the language of some of the arguments to come, we introduce a
map \label{ntn:sigmax}
$\sigma : \mathcal{B}^\epsilon \to \mathscr{B}^*$ by setting
$\sigma(x) = s$ if $x \in \Q^s$.  We adopt the convention that if $x$
is on a boundary component of a quad-T, and so belongs to two
different quad-T's, $\sigma(x)$ gives the shorter label.  For example,
if $x \in \Q^{101} \cap \Q^{1011}$, then $\sigma(x) = 101$.

\begin{lemma}
  \label{lemma:increasing_sigmas}
  If $\left( x_n \right)_{n \in \mathbb{N}}$ is a Cauchy sequence in
  $\tsurf$ which does not converge to a point of $\mathcal{B}^\epsilon$, then
  $|\sigma(x_n)| \to \infty$ as $n \to \infty$.  Passing to a
  subsequence, we may assume that $|\sigma(x_n)|$ is strictly
  increasing.
\end{lemma}

\begin{proof}

  If $\left( x_n \right)_{n \in \mathbb{N}}$ is a Cauchy sequence in
  $\tsurf$ which does not converge in $\mathcal{B}^\epsilon$, it
  cannot be contained in any $\mathcal{B}^\epsilon_n$ (since
  $\mathcal{B}^\epsilon$ is complete) and so
  $\sup |\sigma(x_n)| = \infty$.  As the sequence
  $(x_n)_{n\in \mathbb{N}}$ is Cauchy, we must have
  $|\sigma(x_n)| \to \infty$.
\end{proof}  

In Proposition \ref{prop:elusive_binary}, we will show that there is a
1--1 correspondence between
$\esurf=\tsurf\setminus\mathcal{B}^\epsilon$ and the set of all
infinite binary strings.  Lemma \ref{lemma:increasing_sigmas} showed
that every Cauchy sequence $(x_n)_{n\in\N}$ not convergent in
$\mathcal{B}^\epsilon$, $(|\sigma(x_n)|)_{n\in\N}$ can be thought of
as a strictly increasing sequence, but did not necessarily conclude
that $|\sigma(x_{n+1})|= |\sigma(x_{n})|+1$, for all $n\in\N$.

In order for us to show the desired correspondence in Proposition \ref{prop:elusive_binary}, we must show that there exists a Cauchy sequence $(y_n)_{n\in\N}$ in $\tsurfr$ equivalent to a Cauchy sequence $(x_n)_{n\in\N}$ guaranteed by Lemma \ref{lemma:increasing_sigmas} and $|\sigma(y_n+1)| = |\sigma(y_n)|+1$, for every $n\in \N$.  We state this as Lemma \ref{lemma:cauchy_strings}.

\begin{lemma}
  \label{lemma:cauchy_strings}
  If $\left( x_n \right)_{n \in \mathbb{N}}$ is a Cauchy sequence in
  $\tsurf$ which does not converge to a point of
  $\mathcal{B}^\epsilon$, then there exists an equivalent Cauchy
  sequence $\left( y_n \right)_{n \in \mathbb{N}_0}$ of $\tsurf$
  where for each $n$, $|\sigma(y_n)| = n$ and
  $\sigma(y_n) \wedge \sigma(y_{n+1}) = \sigma(y_n)$.
\end{lemma}
\begin{proof}
  By Lemma~\ref{lemma:increasing_sigmas} we may assume that
  $\left(|\sigma(x_n)|\right)_{n\in\mathbb{N}}$ is a strictly
  increasing sequence.  Let $j \in \mathbb{N}_0$.  Since
  $(x_n)_{n\in\mathbb{N}}$ is a Cauchy sequence in $\tsurf$, there
  exists $N_j \in \mathbb{N}$ such that for all $m, n > N_j$ we have
  $d(x_m, x_n) < 2^{-j}$.  We must then have that for each
  $m, n > N_j$,
  \[
    \left| \sigma(x_m) \wedge \sigma(x_n) \right| > j
  \]
  and so the first $j$ characters of $\sigma(x_m)$ and $\sigma(x_n)$
  must agree.  
  
  Now for each $j$ let $\sigma_j$ be a string of $j$ bits
  agreeing with $\sigma(x_m) \wedge \sigma(x_n)$ for $m, n > N_j$.
  Now choose $y_j$ to be any point in the interior of $\Q^{\sigma_j}$.  By
  construction, $\sigma_j = \sigma(y_j)$ and $|\sigma_j| = j$.  Therefore, 
  \[\sigma(y_{j+1})\wedge \sigma(y_j) = \sigma(y_j)\] 
  \noindent and $|\sigma(y_{j+1})| = j+1 = |\sigma(y_j)|+1$, for all
  $j$, meaning that $\sigma(y_{j+1}) = \sigma(y_j)s'$, where $s'=0$ or
  $s'=1$.  By construction, for every $n\geq N_j$,
  $\sigma(y_i)\wedge \sigma(x_n) = \sigma(y_j)$, resulting in
  $x_n\in \mathcal{B}^{\sigma_j}$. By
  Lemma~\ref{lemma:branch_diameter}, for $n > N_j$, $d(x_n, y_j)$ is
  at most $3\cdot 2^{-|\sigma_j|}$.  Hence, the distance between
  points of the $\left( x_n \right)_{n \in \mathbb{N}}$ sequence and
  the $\left( y_j \right)_{j \in \mathbb{N}}$ sequence goes to zero
  and the two sequences determine the same point in the metric
  completion $\tsurf$.
\end{proof}


\begin{proposition}
  \label{prop:uniqueAddressOfElusivePoint}
  The address $\alpha(x)$ of $x\in \esurf$ is unique.
\end{proposition}

\begin{proof}
  Suppose $x,y \in \esurf$ and suppose that
  $\left(x_n\right)_{n \in \N_0}$ and $\left(y_n\right)_{n \in \N_0}$
  are Cauchy sequences converging to $x$ and $y$, respectively, where
  for each $n$, $|\sigma(x_n)| = n$,
  $\sigma(x_n) \wedge \sigma(x_{n+1}) = \sigma(x_n)$, and likewise for
  the $\sigma(y_n)$.

  If $x = y$, then $\sigma(x_n)$ and $\sigma(y_n)$ must be equal for
  each $n$: if not, say $\sigma(x_{n_0}) \neq \sigma(y_{n_0})$.  Since
  $x_{n_0}\in \Q^{\sigma(x_{n_0})}$ and
  $y_{n_0}\in \Q^{\sigma(y_{n_0})}$ and
  Lemma~\ref{lemma:qt_lowerbound} states that
  $\dist( \Q^{\sigma(x_{n_0})}, \Q^{\sigma(y_{n_0})})$ is bounded
  below by a positive constant, we have that
  \[
    d(x_n,y_n)\geq \dist( \Q^{\sigma(x_{n_0})},  \Q^{\sigma(y_{n_0})})
  \]
  \noindent for all $n\geq n_0$, we must have that $d(x_n, y_n)$ is
  bounded below for each $n > n_0$.  This means that the sequences
  $\left(x_n\right)_{n \in \N_0}$ and $\left(y_n\right)_{n \in \N_0}$
  determine different points of $\esurf$.

  If $\alpha(x) = \alpha(y)$, then $\sigma(x_n) = \sigma(y_n)$ for all
  $n$.  Consequently, $x_n$ and $y_n$ are always in the same quad-T
  subsurface, $\Q^{\sigma(x_n)} = \Q^{\sigma(y_n)}$.  By
  Lemma~\ref{lemma:qt_upperbound}, $d(x_n, y_n) \leq 5 \cdot 2^{-n}$,
  and so $x = y$.
\end{proof}
\begin{proposition}
  \label{prop:elusive_binary}
  The points of $\esurf = \tsurf \setminus \mathcal{B}^\epsilon$ are
  in one-to-one correspondence with the set of all infinite binary
  strings.
\end{proposition}
\begin{proof}
  By Lemma~\ref{lemma:cauchy_strings}, each point $y\in \esurf$ can be
  described as the limit of a Cauchy sequence
  $\left( y_j \right)_{j \in \mathbb{N}}$ where
  $\sigma(y_0) = \epsilon$ and $\sigma(y_{j+1})$ is obtained by
  appending a single bit to $\sigma(y_j)$; the address of $y$ is then
  $\alpha(y)$.  Given an infinite string
  $\beta = (b_i)_{i\in \N_{0}}$, with $b_0 = \epsilon$ and $b_i=0$ or
  $b_i=1$ for $i\geq 1$, let $\beta_n\coloneqq (b_i)_{i=0}^n$. Then
  $\Q^{\beta_n}$ determines a quad-T subsurface of $\tsurf$.  For each
  $n$, choose $x\in \Q^{\beta_n}$ so that $\sigma(x_n) = \beta_n$.  By
  construction, $|\sigma(x_n)|$ is a strictly increasing sequence and
  $(x_n)_{n\in\N_0}$ must converge to some point of $\esurf$.
  Otherwise, $x_n\to x$ where $x\not\in \esurf$ would imply that there
  exists $m\in\N$ such that $\sigma(x_m) = \sigma(x_n)$ for all
  $n\geq m$, further implying that $(|\sigma(x_n)|)_{n\in\N_0}$ was
  not a strictly increasing sequence.
\end{proof}

There are two natural metrics on $\esurf$ we may consider.  Perhaps
the most natural metric for $\esurf$ is the metric of $\tsurf$
restricted to $\esurf$.  By Proposition~\ref{prop:elusive_binary} we
may also identify each elusive singularity with an infinite binary
string, and under this identification it is natural to consider the
2-adic metric on $\esurf$, which we denote $d_2$.

\begin{definition}
  \label{definition:addressOfElusiveSingularityANDtwo-adicMetric}
  Let $x \in \esurf$ and $\alpha(x)$ be the address of $x$.  We define the
  \emph{2-adic metric} on $\esurf$ to be the 2-adic metric on the set
  of binary strings,
  \[
    d_2(x, y) = 2^{-|\alpha(x) \wedge \alpha(y)|}.
  \]
\end{definition}

\begin{theorem}
  \label{thm:equivalent_metrics}
  If $d$ is the metric defined on $\tsurf$, then metric spaces
  $(\esurf,d)$ and $(\esurf,d_2)$ are equivalent metric spaces.
\end{theorem}

\begin{proof}
  Let $x,y\in \esurf$ and $(x_n)_{n\in\N_0}$, $(y_n)_{n\in\N_0}$ be
  Cauchy sequences where $(|\sigma(x_n)|)_{n\in\N_0}$ and
  $(|\sigma(y_n)|)_{n\in\N_0}$ are strictly increasing sequences.  If
  $x=y$, then $d_2(x,y) = d(x,y) = 0$.  Otherwise, there exists
  $n\in\N_0$ such that
  $\alpha(x) \wedge \alpha(y) = \sigma(x_n)\wedge \sigma(y_n)$.  Since
  for every $m>n$, $x_m$ and $y_m$ are in the same branch
  $\mathcal{B}^{\sigma(x_n)\wedge \sigma(y_n)}$, it follows from Lemma
  \ref{lemma:qt_lowerbound} that
  \[
    d(x_m,y_m)\geq 2^{-|\sigma(x_n)\wedge \sigma(y_n)|} = 2^{-n}
  \]
  \noindent for every $m > n$.  By Lemma \ref{lemma:qt_upperbound}, 
  \[
    d(x_m,y_m) \leq 15\cdot 2^{-n} = 15\cdot d_2(x,y).
  \]
  Therefore, 
  \[
    d_2(x, y)
    \leq d(x_m, y_m)
    \leq  15 \cdot d_2(x, y).
  \]
  \noindent for every $m>n$.  Letting $m\to \infty$,
  \[
    d_2(x, y)
    \leq d(x, y)
    \leq 15 \cdot d_2(x, y).
  \]

  Thus $d$ and $d_2$ define equivalent metrics on $\esurf$.

\end{proof}

\begin{theorem}
  \label{prop:elusive_cantor}
  The set of elusive singularities $\esurf$, using either the metric
  $d$ from $\tsurf$ or the 2-adic metric $d_2$, is a Cantor set of
  Hausdorff dimension $1$.
\end{theorem}
\begin{proof}
  Following immediately as a corollary to Theorem
  \ref{thm:equivalent_metrics}, the identity map between these two
  metric spaces, $(\esurf, d)$ and $(\esurf, d_2)$, is a bi-Lipschitz
  map and so preserves Hausdorff dimension.  By
  Proposition~\ref{prop:elusive_binary}, $(\esurf, d_2)$ is isometric
  to the set of 2-adic integers.  Since the $2$-adic integers form a
  totally disconnected perfect set, it follows that $\esurf$ is also
  totally disconnected and perfect, hence, a Cantor set.  Since the
  2-adic integers have Hausdorff dimension $1$ and $(\esurf,d_2)$ is
  isometric with the 2-adic integers, it follows that the Hausdorff of
  $\esurf$, with respect to either $d_2$ or $d$ inherited from
  $\tsurf$ is also equal to $1$.
\end{proof}

\begin{theorem}
  The metric completion $\tsurf$ of the surface $\tsurfr$ is not a
  surface.
\end{theorem}
\begin{proof}
  If $\tsurf$ were a surface, then every point would be contained in
  some chart domain homeomorphic to an open subset of the plane.  In
  particular, for every point there would exist some $\epsilon > 0$ so
  that the $\epsilon$-ball centered at that point would be
  homeomorphic to a disc.  We show this is not the case for elusive
  points by noting that every $\epsilon$-ball around an elusive point
  contains a branch of the $T$-fractal: for each elusive singularity $x
  \in \esurf$ and each $\epsilon > 0$, there exists a binary string $s
  \in \mathscr{B}^*$ such that $\mathcal{B}^s \subseteq
  B_\epsilon(x)$.  We now note that $B_\epsilon(x)$ has a non-trivial
  first homology group: consider a vertical, geodesic loop $\mu$ which
  passes through two quad-T subsurfaces in $\mathcal{B}^s$ as shown in
  Figure~\ref{fig:nontrivial}.  The space $B_\epsilon(x) \setminus
  \mu$ remains path-connected, implying $\mu$ is homologically
  non-trivial.  Hence, for every $\epsilon > 0$, $H_1(B_\epsilon(x)) \neq
  0$. Consequently $B_\epsilon(x)$ is not homeomorphic to a disc, so
  $\tsurf$ is not a surface.

  \begin{figure}[h!]
    \centering
    \begin{tikzpicture}[scale=0.66]
      \begin{scope}[thick,xshift=1.25in]
        \draw (0, 0) -- (0, 2) 
        -- (-1, 2)
        -- (-1, 3);
        \draw (0, 3) -- (2, 3);
        \draw (3, 3) -- (3, 2)
        -- (2, 2)
        -- (2, 0);

        \draw (-1, 3) -- (-1, 4) -- (-1.5, 4) -- (-1.5, 4.5);
        \draw[dotted] (-1.5, 4.5) -- (-1, 4.5);
        \draw (-1, 4.5) -- (0, 4.5);
        \draw[dotted] (0, 4.5) -- (0.5, 4.5);
        \draw (0.5, 4.5) -- (0.5, 4) -- (0, 4) -- (0, 3);

        \draw (2, 3) -- (2, 4) -- (1.5, 4) -- (1.5, 4.5);
        \draw[dotted] (1.5, 4.5) -- (2, 4.5);
        \draw (2, 4.5) -- (3, 4.5);
        \draw[dotted] (3, 4.5) -- (3.5, 4.5);
        \draw (3.5, 4.5) -- (3.5, 4) -- (3, 4) -- (3, 3);

        \draw[black!50!white] (2.5, 2) --
        node[right,font=\tiny] {$\mu$} (2.5, 4.5);

        \draw[dotted] (0, 0) -- (2, 0);

      \end{scope}

      \begin{scope}[thick,xshift=1.25in,yscale=-1, yshift=0.5in]
        \draw (0, 0) -- (0, 2)
        --  (-1, 2)
        -- (-1, 3);
        \draw (0, 3) -- (2, 3);
        \draw (3, 3) -- (3, 2)
        -- (2, 2)
        -- (2, 0);

        \draw (-1, 3) -- (-1, 4) -- (-1.5, 4) -- (-1.5, 4.5);
        \draw[dotted] (-1.5, 4.5) -- (-1, 4.5);
        \draw (-1, 4.5) -- (0, 4.5);
        \draw[dotted] (0, 4.5) -- (0.5, 4.5);
        \draw (0.5, 4.5) -- (0.5, 4) -- (0, 4) -- (0, 3);

        \draw (2, 3) -- (2, 4) -- (1.5, 4) -- (1.5, 4.5);
        \draw[dotted] (1.5, 4.5) -- (2, 4.5);
        \draw (2, 4.5) -- (3, 4.5);
        \draw[dotted] (3, 4.5) -- (3.5, 4.5);
        \draw (3.5, 4.5) -- (3.5, 4) -- (3, 4) -- (3, 3);

        \draw[black!50!white] (2.5, 2) -- node[right,font=\tiny]
        {$\mu$} (2.5, 4.5);

        \draw[black!30!white] (2.75, 2.25) -- (3, 2.25);
        \draw[black!30!white] (-1, 2.25) -- (2.25, 2.25);
        \filldraw[black!30!white] (2.25, 2.25) circle (0.01in);
        \filldraw[black!30!white] (2.75, 2.25) circle (0.01in);

        \draw[dotted] (0, 0) -- (2, 0);
      \end{scope}
      
      \begin{scope}[thick,xshift=-1.25in]
        \draw (0, 0) -- (0, 2) 
        -- (-1, 2)
        -- (-1, 3);
        \draw (0, 3) -- (2, 3);
        \draw (3, 3) -- (3, 2)
        -- (2, 2)
        -- (2, 0);

        \draw (-1, 3) -- (-1, 4) -- (-1.5, 4) -- (-1.5, 4.5);
        \draw[dotted] (-1.5, 4.5) -- (-1, 4.5);
        \draw (-1, 4.5) -- (0, 4.5);
        \draw[dotted] (0, 4.5) -- (0.5, 4.5);
        \draw (0.5, 4.5) -- (0.5, 4) -- (0, 4) -- (0, 3);

        \draw (2, 3) -- (2, 4) -- (1.5, 4) -- (1.5, 4.5);
        \draw[dotted] (1.5, 4.5) -- (2, 4.5);
        \draw (2, 4.5) -- (3, 4.5);
        \draw[dotted] (3, 4.5) -- (3.5, 4.5);
        \draw (3.5, 4.5) -- (3.5, 4) -- (3, 4) -- (3, 3);

        \draw[dotted] (0, 0) -- (2, 0);

      \end{scope}
      \begin{scope}[thick,xshift=-1.25in,yscale=-1,yshift=0.5in]
        \draw (0, 0) -- (0, 2) 
        -- (-1, 2)
        -- (-1, 3);
        \draw (0, 3) -- (2, 3);
        \draw (3, 3) -- (3, 2)
        -- (2, 2)
        -- (2, 0);

        \draw (-1, 3) -- (-1, 4) -- (-1.5, 4) -- (-1.5, 4.5);
        \draw[dotted] (-1.5, 4.5) -- (-1, 4.5);
        \draw (-1, 4.5) -- (0, 4.5);
        \draw[dotted] (0, 4.5) -- (0.5, 4.5);
        \draw (0.5, 4.5) -- (0.5, 4) -- (0, 4) -- (0, 3);

        \draw (2, 3) -- (2, 4) -- (1.5, 4) -- (1.5, 4.5);
        \draw[dotted] (1.5, 4.5) -- (2, 4.5);
        \draw (2, 4.5) -- (3, 4.5);
        \draw[dotted] (3, 4.5) -- (3.5, 4.5);
        \draw (3.5, 4.5) -- (3.5, 4) -- (3, 4) -- (3, 3);

        \draw[black!30!white] (-1, 2.25) -- (3, 2.25);

        \draw[dotted] (0, 0) -- (2, 0);

      \end{scope}
    \end{tikzpicture}
    \caption{A homologically non-trivial curve which passes through
      two quad-T subsurfaces inside any $\epsilon$-ball around an
      elusive point.  The grey, horizontal curve shows that the space
      remains path-connected even when $\gamma$ is removed.  Only the
      portions of the quad-T subsurface intersected by the curve are
      shown.}
    \label{fig:nontrivial}
  \end{figure}
\end{proof}

\section{Elusive Singularities are Wild Singularities}
\label{sec:elusiveSingularitiesWildSingularities}
In this section we show that each elusive singularity of the $T$-fractal
is a wild singularity.

Recall from
Definition~\ref{definition:addressOfElusiveSingularityANDtwo-adicMetric}
that $\alpha(x)$ is the address of the elusive singularity $x$.

\begin{lemma}
  \label{lemma:limitpoints}
  Every elusive singularity $x$ is a limit point of the set of 
  conical singularities of $\tsurfr$.
\end{lemma}
\begin{proof}
  Suppose that $x$ is an elusive singularity with address
  $\alpha(x) = (\alpha_n)_{n \in \mathbb{N}}$ and let $\epsilon > 0$
  be given.  Choose $k > 0$ such that $15 \cdot 2^{-k} < \epsilon$. 
  Let
  $s$ be the string $s = \alpha_1 \alpha_2 \ldots \alpha_k$.  The
  quad-T subsurfaces of the branch $\mathcal{B}^s$ of the $T$-fractal
  translation surface are then within $\epsilon$-distance of $x$, and hence,
  so are the conical singularities of those subsurfaces.
\end{proof}

\begin{theorem}
  Every elusive singularity of the $T$-fractal surface is a wild
  singularity.
\end{theorem}
\begin{proof}
  We simply need to show that each elusive singularity cannot be a
  conical singularity of either finite or infinite angle.

  Suppose $x \in \esurf$ had a rotational component isometric to a
  circle.  It would then be possible to embed a punctured disc in
  $\tsurfr$ centered at $x$.  However, this is impossible by
  Lemma~\ref{lemma:limitpoints} as any neighborhood around an elusive
  singularity must contain conical singularities.  Thus no rotational
  component of $x$ is isometric to a circle, and so $x$ cannot be a
  finite angle conical singularity.

  Similarly, if an elusive singularity were an infinite angle conical
  singularity, then a punctured neighborhood of the point in
  $\mathring{T}$ would be an infinite cyclic cover of the disc.
  However, by Lemma~\ref{lemma:limitpoints} this cannot be the case
  since any neighborhood of an elusive singularity contains
  infinitely-many conical singularities.
\end{proof}

In \cite{BowmanValdez}, Bowman and Valdez made the explicit assumption
that the singularity set of a translation surface is discrete to rule
out certain pathological examples.  The Cantor set of singularities on
the $T$-fractal surface is of course not discrete, and so it is
conceivable that some typical notions associated with wild
singularities, such as linear approaches and rotational components,
are not well-defined or at least not interesting for the $T$-fractal
surface.

\begin{lemma}
  Every elusive singularity has infinitely-many rotational components.
\end{lemma}
\begin{proof}
  To prove this we will consider the action of a particular symmetry
  of the surface $\tsurfr$ on linear approaches to elusive
  singularities.  Notice that each quad-T subsurface has four
  horizontal cylinders as shown in Figure~\ref{fig:cylinders}.
  \begin{figure}[h!]
    \centering
    \begin{tikzpicture}[scale=0.66]
      \begin{scope}[thick,xshift=1.25in]
        \filldraw[black!15!white] (-1, 2) rectangle (3, 3);
        \filldraw[black!30!white] (0, 0) rectangle (2, 2);
        
        \draw (0, 0) -- node[scale=0.66,left] {$E$} (0, 2) 
        -- node[scale=0.66,below] {$F$} (-1, 2)
        -- node[scale=0.66,left] {$G$} (-1, 3);
        \draw (0, 3) -- node[scale=0.66,above] {$A$} (2, 3);
        \draw (3, 3) -- node[scale=0.66,right] {$B$} (3, 2)
        -- node[scale=0.66,below] {$C$} (2, 2)
        -- node[scale=0.66,right] {$D$} (2, 0);

        \draw[\sigmaColor] (-1, 3) -- node[above] {$\sigma_4$} (0, 3);
        \draw[\sigmaColor] (2, 3) -- node[above] {$\sigma_6$} (3, 3);
        \draw[\gammaColor] (2, 0) -- node[below] {$\gamma_5$} (0, 0);
      \end{scope}

      \begin{scope}[thick,xshift=-1.25in]
        \filldraw[black!15!white] (-1, 2) rectangle (3, 3);
        \filldraw[black!30!white] (0, 0) rectangle (2, 2);

        \draw (0, 0) -- node[scale=0.66,left] {$D$} (0, 2)
        -- node[scale=0.66,below] {$J$} (-1, 2)
        -- node[scale=0.66,left] {$B$} (-1, 3);
        \draw (0, 3) -- node[scale=0.66,above] {$H$} (2, 3);
        \draw (3, 3) -- node[scale=0.66,right] {$G$} (3, 2)
        -- node[scale=0.66,below] {$I$} (2, 2)
        -- node[scale=0.66,right] {$E$} (2, 0);

        \draw[\sigmaColor] (-1, 3) -- node[above] {$\sigma_1$}(0, 3);
        \draw[\sigmaColor] (2, 3) -- node[above] {$\sigma_3$} (3, 3);
        \draw[\gammaColor] (2, 0) -- node[below] {$\gamma_2$} (0, 0);
      \end{scope}

      \begin{scope}[thick,xshift=-1.25in,yscale=-1, yshift=0.5in]
        \filldraw[black!60!white] (-1, 2) rectangle (3, 3);
        \filldraw[black!45!white] (0, 0) rectangle (2, 2);

        \draw (0, 0) -- node[scale=0.66,left] {$N$} (0, 2)
        -- node[scale=0.66,above] {$J$} (-1, 2)
        -- node[scale=0.66,left] {$M$} (-1, 3);
        \draw (0, 3) -- node[scale=0.66,below] {$H$} (2, 3);
        \draw (3, 3) -- node[scale=0.66,right] {$L$} (3, 2)
        -- node[scale=0.66,above] {$I$} (2, 2)
        -- node[scale=0.66,right] {$K$} (2, 0);

        \draw[\gammaColor] (-1, 3) -- node[below] {$\gamma_1$} (0, 3);
        \draw[\gammaColor] (2, 3) -- node[below] {$\gamma_3$} (3, 3);
        \draw[\sigmaColor] (2, 0) -- node[above] {$\sigma_2$} (0, 0);
      \end{scope}
      
      \begin{scope}[thick,xshift=1.25in,yscale=-1, yshift=0.5in]
        \filldraw[black!60!white] (-1, 2) rectangle (3, 3);
        \filldraw[black!45!white] (0, 0) rectangle (2, 2);

        \draw (0, 0) -- node[scale=0.66,left] {$K$} (0, 2)
        -- node[scale=0.66,above] {$F$} (-1, 2)
        -- node[scale=0.66,left] {$L$} (-1, 3);
        \draw (0, 3) -- node[scale=0.66,below] {$A$} (2, 3);
        \draw (3, 3) -- node[scale=0.66,right] {$M$} (3, 2)
        -- node[scale=0.66,above] {$C$} (2, 2)
        -- node[scale=0.66,right] {$N$} (2, 0);

        \draw[\gammaColor] (-1, 3) -- node[below] {$\gamma_4$} (0, 3);
        \draw[\gammaColor] (2, 3) -- node[below] {$\gamma_6$} (3, 3);
        \draw[\sigmaColor] (2, 0) -- node[above] {$\sigma_5$} (0, 0);
      \end{scope}
    \end{tikzpicture}
    \caption{Each quad-T subsurface is built from four horizontal cylinders.}
    \label{fig:cylinders}
  \end{figure}
  Two of these cylinders have dimensions $4 \times 2$, and so have
  modulus $2$; and the other two cylinders have dimensions $8 \times
  1$ with modulus $8$.  Thus the affine diffeomorphism with derivative
  \[
    D = \begin{pmatrix} 1 & 8 \\ 0 & 1 \end{pmatrix}
  \]
  acts by twisting these cylinders in such a way that the horizontal
  foliation in each cylinder is preserved, but the boundaries of the
  cylinders are fixed pointwise.  Since this is true for each quad-T
  subsurface, there exists some well-defined affine diffeomorphism
  $\phi : \tsurfr \to \tsurfr$ with derivative $D$.

  Let $\mu$ be any linear approach to an elusive singularity.  Since
  $\phi$ fixes the boundary components of each horizontal cylinder in
  each quad-T subsurface, $\phi(\mu)$ is another linear approach to
  the same elusive singularity.  However, because $\phi$ twists each
  cylinder (the $4 \times 2$ cylinders are twisted four times, and the
  $8 \times 1$ cylinders are twisted once), $\mu$ and $\phi(\gamma)$
  intersect in each quad-T containing $\mu$ (and hence, $\phi(\mu)$),
  as indicated in Figure~\ref{fig:twists}.
  \begin{figure}[h!]
    \centering
    \begin{tikzpicture}[scale=0.66]
      \begin{scope}[thick,xshift=1.25in]
        \draw[->] (0.5, 0) -- (2.25, 3); 
        
        \draw[->,gray] (0.5, 0) -- (2, 0.1714); 
        \draw[->,gray] (0, 0.4) -- (2, 0.6286);
        \draw[->,gray] (0, 0.8572) -- (2, 1.0858);
        \draw[->,gray] (0, 1.3144) -- (2, 1.543);
        \draw[gray] (0,1.7716) -- (1.6667, 2);
        \draw[->,gray] (1.6667, 2) -- (3, 2.1524);
        \draw[->,gray] (-1, 2.6096) -- (2.25, 3);
        
        \draw (0, 0) -- (0, 2) 
        -- (-1, 2)
        -- (-1, 3);
        \draw (0, 3) -- (2, 3);
        \draw (3, 3) -- (3, 2)
        -- (2, 2)
        -- (2, 0);

        \draw[\sigmaColor] (-1, 3) -- (0, 3);
        \draw[\sigmaColor] (2, 3) -- (3, 3);
        \draw[\gammaColor] (2, 0) -- (0, 0);
      \end{scope}

      \begin{scope}[thick,xshift=-1.25in]
        \draw[->,gray] (0, 0.1714) -- (2, 0.4);
        \draw[->,gray] (0, 0.6286) -- (2, 0.8572);
        \draw[->,gray] (0, 1.0858) -- (2, 1.3144);
        \draw[->,gray] (0, 1.543) -- (2, 1.7716);
        \draw[->,gray] (-1, 2.1524) -- (3, 2.6096);
        
        \draw (0, 0) -- (0, 2)
        -- (-1, 2)
        -- (-1, 3);
        \draw (0, 3) -- (2, 3);
        \draw (3, 3) -- (3, 2)
        -- (2, 2)
        -- (2, 0);

        \draw[\sigmaColor] (-1, 3) -- (0, 3);
        \draw[\sigmaColor] (2, 3) -- (3, 3);
        \draw[\gammaColor] (2, 0) -- (0, 0);
      \end{scope}

      \begin{scope}[thick,xshift=-1.25in,yscale=-1, yshift=0.5in]
        \draw (0, 0) -- (0, 2)
        -- (-1, 2)
        -- (-1, 3);
        \draw (0, 3) -- (2, 3);
        \draw (3, 3) -- (3, 2)
        -- (2, 2)
        -- (2, 0);

        \draw[\gammaColor] (-1, 3) -- (0, 3);
        \draw[\gammaColor] (2, 3) -- (3, 3);
        \draw[\sigmaColor] (2, 0) -- (0, 0);
      \end{scope}
      
      \begin{scope}[thick,xshift=1.25in,yscale=-1, yshift=0.5in]
        \draw (0, 0) -- (0, 2)
        -- (-1, 2)
        -- (-1, 3);
        \draw (0, 3) -- (2, 3);
        \draw (3, 3) -- (3, 2)
        -- (2, 2)
        -- (2, 0);

        \draw[\gammaColor] (-1, 3) -- (0, 3);
        \draw[\gammaColor] (2, 3) -- (3, 3);
        \draw[\sigmaColor] (2, 0) -- (0, 0);
      \end{scope}
    \end{tikzpicture}
    \caption{Given a geodesic $\mu$ (the dark curve in this figure)
      we construct a new linear approach $\phi(\mu)$ (the lighter
      curve), which must pass through the same sequence of quad-T's
      since $\phi$ preserves the boundary each quad-T.  If $\mu$ is
      a linear approach to an elusive singularity, then $\phi(\mu)$
      is another linear approach.}
    \label{fig:twists}
  \end{figure}
  This means that $\mu$ and $\phi(\mu)$ cannot be rotationally
  equivalent.  Repeating this process by iterating $\phi$ generates a
  sequence of linear approaches to the elusive singularity, $\mu$,
  $\phi(\gamma)$, $\phi^2(\gamma)$, $\ldots$, each of which is in a
  different rotational component than the other linear approaches in
  the sequence.  Hence, the elusive singularity has infinitely-many
  rotational components.
\end{proof}

We now show that only elusive singularities with rational addresses
may admit rotational components of positive cone angle.

\begin{theorem}
  \label{thm:positive_angle_rational}
  If an elusive singularity of the T-fractal surface has a rotational
  component with a positive cone angle, then the elusive singularity
  has a rational address.
\end{theorem}

Before giving the proof of Theorem~\ref{thm:positive_angle_rational},
we discuss the strategy of the proof for the convenience of the
reader.  We will show that each linear approach in a rotational component of
positive cone angle is ``periodic,'' in the sense that it must pass
through scaled copies of the same point on the boundaries of quad-T
subsurfaces.  That is, if $\gamma$ is such a linear approach, we
will show there exists some point $x$ on $\partial \mathcal{Q}$ so
$\gamma$ goes through $x^{s}$ and $x^{s'}$ where these are copies of
$x$ on the boundaries of quad-T's $\mathcal{Q}^s$ and
$\mathcal{Q}^{s'}$ which $\gamma$ intersects.  Once the existence of
such $x^s$ and $x^{s'}$ is established, we see that the sequence of
$1$'s and $0$'s which are appended to $s$ as $\gamma$ moves from $x^s$
to $x^{s'}$ must be repeated before $\gamma$ intersects some point
$x^{s''}$ which is again a copy of $x$ on a scaled quad-T
$\mathcal{Q}^{s''}$.  This process repeats \emph{ad infinitum} giving a
repeating address for the elusive singularity, and so the address is
rational.

In order to establish the existence of $x^s$ and $x^{s'}$, we will
consider an embedded triangle of linear approaches to the elusive
singularity.  This triangle intersects infinitely-many boundaries of
quad-T subsurfaces (horizontal saddle connections) in an interval.
Though the lengths of the horizontal saddle connections shrink to zero
as we approach the elusive singularity, we can consider the proportion
of the saddle connection which the triangle of linear approaches
intersects; this is tantamount to renormalizing these horizontal
saddle connections by an affine map to the unit interval and consider
the subintervals which correspond to the intersection with the
aforementioned triangle.  We will show that each of these renormalized
subintervals has the same length (corresponding to the triangle
intersecting the shrinking horizontal saddle connections in the same
proportions -- e.g., $\sfrac{1}{2}$ of each horizontal saddle
connection), and so eventually two of these renormalized subintervals
must intersect.  By considering a certain function defined on these
subintervals, we will show that if two of these subintervals intersect,
they are in fact the same subinterval.  Points on the quad-T
boundaries corresponding to these intersecting, renormalized intervals
are then intersected by the same linear approaches, giving our $x^s$
and $x^{s'}$.
  
\begin{proof}[Proof of Theorem~\ref{thm:positive_angle_rational}]
  Let $e$ be an elusive singularity with some rotational component $R$
  of positive cone angle.  Suppose that $\gamma_1$ and $\gamma_2$ are
  two distinct linear approaches to $e$ contained in $R$.  Without
  loss of generality, we may suppose that $\gamma_1$ and $\gamma_2$
  are emitted from points on some boundary component $\sigma_i^s$ of
  some quad-T subsurface.  Let $\Delta$ denote the Euclidean triangle
  embedded in the T-fractal which has $e$ as a vertex, $\gamma_1$ and
  $\gamma_2$ as edges adjacent to $e$, and whose edge opposite of $e$
  is the boundary of the quad-T subsurface from which $\gamma_1$ and
  $\gamma_2$ are emitted.  Note that by assumption $\Delta$ contains
  no cone points in its interior or on its boundary.  Let $I_0$ denote
  the edge of $\Delta$ opposite $e$, and let $J_0$ denote the
  horizontal saddle connection (the boundary component $\sigma_i^s$)
  containing $I_0$.

  Note that for each point $x$ of $I_0$ there is a linear approach
  to $e$ emitted from $x$, and each of these linear approaches has a
  distinct slope.  Let $m_0 : I_0 \to \mathbb{R}$ denote the function
  which associates to each $x \in I_0$ the reciprocal of the slope of the linear
  approach to $e$ emanating from $x$.  Observe that $m_0$ is the
  non-constant affine map which assigns the reciprocal of the slope of $\gamma_0$ to the
  left-hand endpoint of $I_0$, and the recirprocal of the slope of $\gamma_1$ to the
  right-hand endpoint of $I_1$.

  Every linear approach in $\Delta$ emanating from a point of $I_0$
  must cross infinitely-many quad-T subsurfaces which are scaled
  copies of translation surface with boundary, $\mathcal{Q}$.  As
  $\mathcal{Q}$ has finitely-many boundary components, there is some
  $i \in \{1, 2, ..., 6\}$ so that the scaled boundary components
  $\sigma_i^s$ for quad-T subsurfaces $\mathcal{Q}^s$ are crossed
  infinitely-many times.  Let $J_n$ denote the sequence of these
  boundary components intersected by the linear approaches in
  $\Delta$, and let $I_n = \Delta \cap J_n$ denote the horizontal
  segments which $\Delta$ intersects on these boundary components.

  Just as we defined an affine $m_0 : I_0 \to \mathbb{R}$ which
  records the reciprocal of the slope of linear approaches to $e$ emitted from points of
  $I_0$, we may also similarly define affine maps $m_n : I_n \to
  \mathbb{R}$.  Notice that for each $x \in I_n$ if we have that $x'
  \in I_{n+1}$ is the point along the linear approach which passes
  through $x$, then $m_n(x) = m_{n+1}(x')$.

  For each $n$, consider the affine map $\pi_n : J_n \to [0, 1]$ which
  simply rescales each $J_n$ so that the left-hand endpoint of $J_n$
  is mapped to $0$, and the right-hand endpoint is mapped to $1$.
  Observe that for each $n$ there exists a subtriangle of $\Delta$
  with vertex $e$, sides which are subsegments of $\gamma_1$ and
  $\gamma_2$, and whose edge opposite of $e$ is $I_n$.  As these triangles
  are all similar, we see that each subinterval
  $\pi_n(I_n) \subseteq [0, 1]$ has the same length, which we call
  $\ell$.

  Defined on each interval $\pi_n(I_n)$, consider the function $M_n =
  m_n \circ \pi_n^{-1}\big|_{\pi_n(I_n)}$.  As each $\pi_n(I_n)$ has
  the same length $\ell$ and each $m_n$ records the same sequence of
  reciprocals of slopes of linear approaches in $\Delta$, we see that the $M_n$ are
  all translates of one another; for each $n$ and $n'$ there exists
  some constant $k_{n,n'}$ so that $M_n'(x) = M_n(x + k_{n,n'})$.

  By the assumption that the rotational component $R$ has positive
  length and $\gamma_1$, $\gamma_2$ are distinct linear approaches, we
  have $\ell > 0$.  As such, the intervals $\pi_n(I_n) \subseteq [0,
    1]$ can not all be disjoint.  Let $n_0$ and $n_1$ be chosen so
  that $\pi_{n_0}(I_{n_0})$ and $\pi_{n_1}(I_{n_1})$ intersect in a
  subinterval of positive length, and let $x$ be some point in that
  intersection.  If $x_0 = \pi_{n_0}^{-1}(x)$ and $x_1 =
  \pi_{n_1}^{-1}(x)$, then there is a linear approach in $\Delta$
  which passes through $x_0$ and $x_1$, and so $m_{n_0}(x_0) =
  m_{n_1}(x_1)$ and this means $M_{n_0}(x) = M_{n_1}(x)$.  As
  $M_{n_1}$ is a translate of the non-constant affine function
  $M_{n_0}$, we must have $M_{n_0} = M_{n_1}$ if there is any $x$ on
  which the two functions agree.  This then means $\pi_{n_0}(I_{n_0})
  = \pi_{n_1}(I_{n_1})$, and so $I_{n_0}$ and $I_{n_1}$ are scaled
  copies of the same interval on the boundaries of their respective
  quad-T subsurfaces.  Hence the linear approaches in $\Delta$ which
  pass through points of $I_{n_0}$ must pass through scaled copies of
  the same points on $I_{n_1}$.  Thus all linear approaches in
  $\Delta$ repeatedly pass these same points on these scaled copies,
  meaning the sequence of zeros and ones appeneded in constructing the
  address of the elusive singularity repeat, and so $e$ has a rational
  address. 
\end{proof}

\begin{corollary}
  \label{cor:ae_rotational_zero_angle}
  The rotational components of almost every elusive singularity of the
  T-fractal surface have zero cone angle.
\end{corollary}

Theorem~\ref{thm:positive_angle_rational} above shows that if a
rotational component has positive cone angle, its elusive singularity must
have rational address.  The theorem does not show, however, that all
rotational components of elusive singularities with rational addresses
necessarily have positive cone angle.  The authors conjecture that this is
true, and one interesting problem would be to determine the cone angle of
the rotational components of rational elusive singularities as a
function of the singularity's address.

\section{Final Remarks}
\label{sec:finalRemarks}
There are many questions about the $T$-fractal surface that are still
unanswered.  In particular, determining precisely which directions
admit linear approaches to elusive singularities and what values can
arise as cone angles of rotational components of elusive singularities
with rational address, are two problems which the authors would be
interesting for future research.
Perhaps the most obvious questions, however, are concerned with the
dynamics of flows on the surface.  In a future paper, the authors
study these dynamical questions by considering an infinite interval
exchange transformation which is the first-return map of the flow to a
collection of particular geodesic intervals on the surface.

\subsection*{Acknowledgements}
The authors are grateful to the referee for their insightful comments
and recommendations for this article.




\providecommand{\bysame}{\leavevmode\hbox to3em{\hrulefill}\thinspace}
\providecommand{\MR}{\relax\ifhmode\unskip\space\fi MR }
\providecommand{\MRhref}[2]{%
  \href{http://www.ams.org/mathscinet-getitem?mr=#1}{#2}
}
\providecommand{\href}[2]{#2}

\end{document}